\documentclass[10pt,reqno]{amsart}
\usepackage{tikz}
\usepackage[margin=1in,letterpaper,portrait]{geometry}
\usepackage{ytableau,shuffle}
\usepackage[enableskew,vcentermath]{youngtab}
\usepackage{latexsym,amsmath,amssymb,verbatim,tikz}
\usepackage[pdfencoding=auto, psdextra]{hyperref}

\makeatletter
\DeclareRobustCommand*{\bfseries}{%
	\not@math@alphabet\bfseries\mathbf
	\fontseries\bfdefault\selectfont
	\boldmath
}

\makeatother

\theoremstyle{plain}

\newtheorem{theorem}{Theorem}[section]
\newtheorem{proposition}[theorem]{Proposition}
\newtheorem{prop}[theorem]{Proposition}
\newtheorem{lemma}[theorem]{Lemma}
\newtheorem{corollary}[theorem]{Corollary}

\theoremstyle{definition}
\newtheorem{definition}[theorem]{Definition}

\newtheorem{example}[theorem]{Example}

\newtheorem{remark}[theorem]{Remark}
\newtheorem{observation}[theorem]{Observation}

\numberwithin{equation}{section}

\newcommand{\NN}{{\mathbb {N}}}

\newcommand{\ZZ}{{\mathbb {Z}}}

\newcommand{\sP}{{\mathsf {P}}}
\newcommand{\sQ}{{\mathsf {Q}}}
\newcommand{\sT}{{\mathsf {T}}}

\newcommand{\FC}{{{\rm FC}}}

\newcommand{\Des}{{\operatorname{Des}}}

\newcommand{\shape}{{\operatorname{shape}}}
\newcommand{\SYT}{{\operatorname{SYT}}}

\newcommand{\SDT}{{\operatorname{SDT}}}
\newcommand{\BSYT}{{\operatorname{BSYT}}}
\newcommand{\SSDT}{{\operatorname{SSDT}}}
\newcommand{\DSYT}{{\operatorname{SDT}}}

\newcommand{\bl}{{\operatorname{bl}}}
\newcommand{\ldes}{{\operatorname{ldes}}}
\newcommand{\Neg}{{\operatorname{Neg}}}

\newcommand{\wDes}{{\operatorname{rDes}}}

\newcommand{\X}{{\operatorname{X}}}
\newcommand{\Y}{{\operatorname{Y}}}

\newcommand{\Q}{{\mathcal Q}}

\newcommand{\bj}[1]{{\color{blue}{#1}}}%{#1}
\renewcommand{\H}{\operatorname{Heap}}

\title[Block number and descents of fully-commutative elements]{
Block number, descents %and 
 and Schur positivity\\ 
	of fully commutative elements in $B_n$}
\author{Eli Bagno}
\address{Jerusalem college of technology, Jerusalem, Israel}
\email{bagnoe@jct.ac.il}

\author{Riccardo Biagioli}
\address{Dipartimento di Matematica, Universit\`a di Bologna, Piazza di Porta San Donato 5, 40126 Bologna, Italy}
\email{riccardo.biagioli2@unibo.it}

\author{Fr\'ed\'eric Jouhet}
\address{Institut Camille Jordan, Universit\'e Claude Bernard Lyon 1, 69622 Villeurbanne
Cedex, France}
\email{jouhet@math.univ-lyon1.fr}

\author{Yuval Roichman}
\address{Department of Mathematics, Bar-Ilan University, 
		Ramat-Gan 52900, Israel}
\email{yuvalr@math.biu.ac.il}

\date{\today}

\thanks{E.B, R.B and F.J were partially supported by the Israeli Ministry of Science and Technology, and the French National Scientific Research Center (CNRS), grant PRC 1656, 
Y.R. was partially supported by the Israel Science Foundation, grant no.\ 1970/18.}

\begin{document}

\begin{abstract}
%The quasi-symmetric generating function of the Coxeter descent set of fully commutative elements
% in type $B_n$, $\FC(B_n)$, with prescribed block number, is shown to be  %symmetric and 
%Schur-positive. 
The distribution of Coxeter descents and block number over the set 
of fully commutative elements in the hyperoctahedral group $B_n$, $\FC(B_n)$, is studied in this paper.
%The coefficients 
%An explicit Schur expansion
We prove that the associated Chow quasi-symmetric generating function is equal to a non-negative sum of products of two Schur functions.
%in the Schur expansion 
%is provided and the coefficients are shown to be nonnegative.
%and an explicit combinatorial description of these coefficients is provided.
The proof involves a decomposition of $\FC(B_n)$ into a disjoint union of two-sided Barbash-Vogan combinatorial  
cells, a type $B$ extension of Rubey's descent preserving involution on $321$-avoiding permutations and a detailed study of the intersection of $\FC(B_n)$ with $S_n$-cosets which yields a new decomposition of $\FC(B_n)$ into disjoint subsets called fibers. We also compare two different type $B$ Schur-positivity notions, arising from works of Chow and Poirier. 
\end{abstract}

\maketitle

%\tableofcontents

%%%%%%%%%%%%%%%%%%%%%%%%%%%%%%%%%%%%%%%%%%%%%%%%%
\section{Introduction}\label{sec:intro}

%\todo{\color{blue} YR: The notion of Schur-positivity in the abstract is confusing, motivate the definition of Schur positivity in type B and make a link with the work of Poirier}

\subsection{Outline}

An element $w$ in a Coxeter group $W$ is fully commutative
if any reduced expression for $w$ in Coxeter generators can be obtained from any other using only commutation relations.
%by transposing adjacent pairs of commuting generators. 
%These elements were extensively studied. 
The study of these elements was motivated by
%original context for the appearance of full commutativity is algebraic and relates to the 
generalizations of the Temperley–Lieb algebra to all Coxeter types.  %Indeed, 
Fan~\cite{Fan} and Graham~\cite{Graham} proved %for instance
that for every Coxeter group $W$, the associated %generalized
Temperley–Lieb algebra admits a linear basis indexed by the fully commutative elements in $W$.
 Various combinatorial characterizations, enumeration 
and connections with enriched $P$-partitions 
%(the letter P stands here for“Poset”) 
and Schur’s $Q$-functions were studied in a series of papers by Stembridge~\cite{ST1, ST3, ST4}.
%A decomposition of the set of fully commutative elements as a disjoint union of Kazhdan--Lusztig cells was presented by Green and Losonczy~\cite{GL}.
Compatibility of the 
Kazhdan--Lusztig cell decomposition of a Coxeter group $W$ with the set of fully commutative elements was studied 
by Green and Losonczy~\cite{GL}.

\medskip

The graded ring of quasi-symmetric functions, introduced by Gessel~\cite{Gessel},
has many applications to enumerative combinatorics, as well as to other branches of mathematics; see, e.g., 
\cite[Ch.\ 7]{EC2}.
A quasi-symmetric function is a formal power series $f(x_1,x_2,\dots,)$ of bounded degree such that for each fixed $k$-tuple $(\alpha_1,\ldots,\alpha_k)$ of nonnegative integers, with $k \in \NN$, all the monomials in $f$ of the form $x_{i_1}^{\alpha_1} \cdots x_{i_k}^{\alpha_k}$, where $i_1<i_2<\cdots< i_k$, share the same coefficient. The vector space of all quasi-symmetric functions which are homogeneous of degree $n$, % denoted by ${\mathcal{Q}}sym_n$, 
 has a 
%Basis elements of this vector space are indexed by compositions of $n$, or equivalently by subsets of $[n-1]:=\{1,2,\,\ldots, n-1\}$. A 
distinguished basis  $\{{F}_J \mid J\subseteq [n-1\}\}$, where
%\[
%{\mathcal F}_J:=\sum\limits_{1\le i_1\le i_2\le \cdots i_n\atop
%j\in J\Rightarrow i_j<i_{j+1}} x_{i_1}\cdots x_{i_n},
%\]
$[n-1]:=\{1,2,\,\ldots, n-1\}$ and
$F_J$ is the Gessel fundamental quasi-symmetric 
function indexed by $J\subseteq [n-1]$.
%Note that the fundamental quasi-symmetric functions are generating functions for the $P$-partitions of permutations of $S_n$,
%%. This basis coincides with the fundamental basis of ${\mathcal Q}sym_n$, 
%where $J$ is the descent set.

%\bj{\todo{FJ: maybe should we add a reference above, for instance Gessel}}

\medskip

The  block number of a permutation $\pi$ in $S_n$, which was studied in~\cite{St} as the cardinality of the connectivity set of $\pi$, is equal to the maximal number of summands in an expression of $\pi$ as a direct sum of smaller permutations.  
It was shown recently that the quasi-symmetric generating function of the descent set statistic over
the set of $321$-avoiding permutations with prescribed block number
is Schur-positive~\cite{ABR}. 
Actually the 321-avoiding permutations in the symmetric group $S_n$ are in one-to-one correspondence with the fully commutative elements in the Coxeter group of type $A_{n-1}$~\cite{BJS}. 
%is fully commutative if and only if %in one-to-one correspondence with
%its permutation representation is $321$-avoiding
%permutations in the symmetric group $S_n$
Similarly, the set of fully commutative elements in the Coxeter group of type $B_n$ has an explicit combinatorial description in terms of several forbidden patterns in signed permutations~\cite{ST3}.

\medskip

The concept of quasi-symmetric functions has been extended to Coxeter group of type $B_n$ in two different ways. Chow's construction applies the presentation of $B_n$ as a Coxeter group; Chow's fundamental quasi-symmetric functions are indexed by type $B$ Coxeter descent sets~\cite{Chow}. Poirier's 
construction applies the presentation of $B_n$ as a wreath product, or equivalently as a colored permutation group; Poirier's fundamental basis elements are indexed by signed descent sets~\cite{Po}. For discussion and comparison of these two families of quasi-symmetric functions of type $B$ see~\cite{Petersen, MV4, AAER}.
In the current paper, we study the type $B$ quasi-symmetric functions
determined by the Coxeter descent sets of fully commutative elements in $B_n$. 
It turns out that while Poirier's approach is not useful in this setting, Chow's provides
a nice description.  In particular, we give an explicit expansion of Chow's quasi-symmetric generating functions 
over the subset of fully commutative elements with a prescribed block number in the Coxeter groups of type $B_n$ in terms of Schur functions and show that the coefficients are non-negative.

\subsection{Main results} 
For a positive integer $n$ and $W$ the Coxeter group of type $A_{n-1}$ or $B_n$, let $\FC(W)$ be its subset of fully commutative elements (see Section~\ref{sec:fc} below for precise definitions). For an integer partition $\lambda$, denote by $s_\lambda$ the associated Schur function.
The {\it block number} of a permutation $\pi=[\pi_1,\dots,\pi_n]$ in the symmetric group $S_n$ is defined by
\[
\bl(\pi):=\#\{i \mid (\forall j<i)\ \pi_j\le i\}=1+\#\{1\le i\le n-1 \mid \max(\pi_1, \dots , \pi_i) < \min(\pi_{i+1}, \dots , \pi_n)\}.
\]
%The block number was studied in~\cite{St} (see also references therein) as the
%cardinality of the connectivity set of a permutation.
Let $F_{\Des(\pi)}$ be Gessel's fundamental quasi-symmetric function indexed by the (right) descent set 
$\Des(\pi)$, as defined in Section~\ref{sec:fc}.
%\bj{\todo{add definition ??? partitions, Young diagram and its symbol}}
For a pair of partitions $\lambda,\mu$ such that $\mu\subseteq\lambda$, denote the set of standard Young tableaux of skew shape $\lambda/\mu$ by $\SYT(\lambda/\mu)$ (see~\cite{EC2} for definitions of these classical objects). 
For a standard Young tableau $T\in \SYT(\lambda/\mu)$, let $\ldes(T)$ be the maximal descent of $T$; if the descent set is empty we set $\ldes(T):=0$ (see Section~\ref{sec:background} for more details).

The following Schur-positivity result is a reformulation of~\cite[Theorem 1.2]{ABR}. 
\begin{theorem}\label{thm:ABR}
For any positive integer $n$, we have
\begin{equation}\label{eq:ABR}
\sum\limits_{\pi \in \FC(S_n)} q^{\bl(\pi)} F_{\Des(\pi)}=\sum\limits_{k=0}^{\lfloor n/2\rfloor}
\left(\sum\limits_{j=0}^n  a_{n,k,j} \ q^j\right) s_{(n-k,k)},
\end{equation}
where $s_{\lambda}$ is the Schur function corresponding to the partition $\lambda$ and 
\[
a_{n,k,j}:=\#\{T\in \SYT(n-k,k) \mid \ldes(T)=n-j \},
\]
which is thus non-negative.
\end{theorem}

\medskip

The main goal of the present work is to prove a type $B$ analogue of the above result.
%\red{
The block number of a signed permutation 
$w=[w_1,\dots,w_n]\in B_n$
is defined by
%in the same way it was defined for permutations of $S_n$, namely:
\[
\bl(w):=1+\#\{1\le i\le n-1 \mid \max(w_1, \dots , w_i) < \min(w_{i+1}, \dots , w_n)\}.\]
%}

Let $s_\lambda(x_I)$ be the Schur function in the set of indeterminates indexed by the elements in the ordered set $I$, and let 
$F^B_{\Des_B(w)}$ be Chow's fundamental quasi-symmetric function indexed by the type $B$ (right) descent set $\Des_B(w)$, see Section~\ref{sec:Bn} for detailed definitions.
Consider the natural embedding of $S_n$ as a maximal parabolic subgroup of $B_n$.
\smallskip

Our main result is the following type $B$ analogue of Theorem~\ref{thm:ABR}.

\begin{theorem}\label{thm:main}
For any positive integer $n$, we have
\begin{equation}
\sum\limits_{w \in \FC(B_n)\setminus \FC(S_n)} q^{\bl(w^{-1})} F^B_{\Des_B(w)}=\sum\limits_{k=1}^{n} \left(\sum\limits_{j=0}^n b_{n,k,j} q^j \right)
s_{(k)}(x_1,x_2,\ldots)\ s_{(n-k)}(x_0, x_1,\ldots),
\end{equation}
where 
\[
b_{n,k,j}:=\#\{T\in \SYT((n,k)/(k)) \mid \ldes(T)=n-j \},
\]
which is thus non-negative.
\end{theorem}

To prove Theorem~\ref{thm:main}, we will combine two new explicit decompositions 
of $\FC(B_n)$, one as a disjoint union of fibers (see Theorem~\ref{thm:fibers}), and one as a disjoint union of Barbash-Vogan combinatorial cells (see Theorem~\ref{thm:cells}),  together with an equidistribition phenomenon which takes the following form.

For a subset $J\subseteq \{0,1,\ldots,n\}$ let ${\bf x}^J:=\prod_{i\in J}x_i$.

\begin{theorem}\label{thm:2} For any positive integer $n$ we have the following equidistribution on $\FC(B_n)$:
\[
\sum\limits_{w\in \FC(B_n)} {\bf x}^{\Des_B(w)}{\bf z}^{{\rm Neg}(w)}q^{\bl(w^{-1})}t^{n-\ldes(w^{-1})}=
\sum\limits_{w\in \FC(B_n)} {\bf x}^{\Des_B(w)}{\bf z}^{{\rm Neg}(w)}q^{n-\ldes(w^{-1})}t^{\bl(w^{-1})}.
\]
\end{theorem}

Here, ${\rm Neg}$ and $\ldes$ denote the negative set and the last descent of a signed permutation, respectively (see Section~\ref{sec:Bn} for precise definitions).

\smallskip
%\todo{\bj{Add here another result ? See the todo box i 543 241n page 24.}} 

Finally, we compare the two different notions of type $B$ Schur-positivity, based on Chow's and Poirier's approaches, studied in~\cite{MV4} and~\cite{AAER}, respectively. 
%Chow's type $B$ Schur-positivity,
%introduced and discussed in~\cite{MV4} with the Poirier's  
It is shown that every Poirier type $B$ Schur-positive set is a Chow type $B$ Schur-positive set (see Theorem~\ref{lem:PC}). 
The converse does not hold: 
the set of fully commutative elements in $B_n$, $\FC(B_n)$, is Chow type $B$ Schur-positive but not Poirier type $B$ Schur-positive, as concluded in  Remark~\ref{rem:converse}. 

\smallskip

The paper is organized as follows. In Sections~\ref{sec:background} and~\ref{sec:heaps} we  provide the necessary background:  Section~\ref{sec:background} is devoted to Coxeter groups, fully commutative elements, quasi-symmetric functions associated with the hyperoctahedral group, and the different kinds of tableaux and statistics that will be in use; Section~\ref{sec:heaps} deals with the theory of heaps as defined by Viennot~\cite{ViennotHeaps}. %, which is particularly useful for our purposes.
%regarding fully commutative elements of types $A$ and $B$. 
Using these heaps, we prove in Section~\ref{sec:fibers} our decomposition of the set $\FC(B_n)$ into fibers. 
%which we will define precisely. 
In Section~\ref{sec:Cellular}, we describe the cellular structure of $\FC(B_n)$.
%, which relies on %an RSK correspondence for 
%Barbash-Vogan bijection to
%domino tableaux and %the 
%a decomposition %, %due to 
%analogous to Green and Losonczy~\cite{GL}, 
%of $\FC(B_n)$ into %Kazhdan--Lusztig 
%combinatorial  cells. 
In Section~\ref{sec:equidistribution}, we prove the equidistribution result given in Theorem~\ref{thm:2} above,  using   the results of Section~\ref{sec:fibers} and an involution due to Rubey~\cite{Rubey}.  In Section~\ref{sec:proofMain}  %the proof of 
we prove Theorem~\ref{thm:main} and  Section~\ref{sec:types}
ends the paper with a discussion on the above mentioned two notions of type $B$ Schur-positivity.
%, based on Chow's and Poirier's approaches. 
%, which uses the results of %both previous sections.

%%%%%%%%%%%%%%%%%%%%%%%%%%%%%%%%%%%%%%%%%%%%%%%%%
\section{Background}
%: Coxeter groups, fully commutative elements, and quasi-symmetric functions}
\label{sec:background}

%%%%%%%%%%%%%%%%%%%%%%%%%%%%%%%%%%%%%%%%%%%%%%%%%
\subsection{Coxeter groups and fully commutative elements}~\label{sec:fc}
%%%%%%%%%%%%%%%%%%%%%%%%%%%%%%%%%%%%%%%%%%%%%%%%%
Let $(W,S)$ be a Coxeter system with Coxeter matrix $M=(m_{st})_{s,t \in S}$. We recall that the finite set of generators $S$ is subject only to relations of the form $(st)^{m_{st}}=1$, where $m_{ss}=1$, and  $m_{st}=m_{ts}\geq 2$, for $s\neq t \in S$. If $st$ has infinite order we set $m_{st}=\infty$. These relations can be rewritten more explicitly as $s^2=1$ for all $s\in S$, and   
$$\underbrace{sts\cdots}_{m_{st}}=\underbrace{tst\cdots}_{m_{ts}},$$
where $m_{st} <\infty$. These are the so-called {\em braid relations}. When ${m_{st}}=2$, they are named {\em commutation relations}, $st=ts$. 
This information is encoded in the {\it Dynkin diagram}, which is a graph with one vertex for each $s \in S$ and in which an edge connects two elements $s,t\in S$ if and only if $m_{st} \geq 3$. When $m_{st}>3$, we write the number $m_{st}$ above the edge connecting $s$ and $t$.

\begin{figure}[!ht]
\begin{center}
\includegraphics{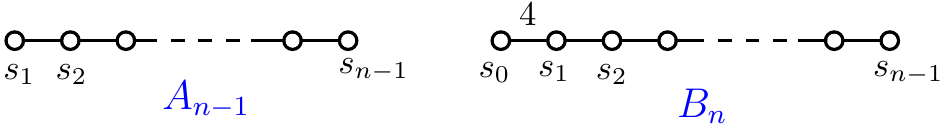}
\caption{\label{AB-types} The  Dynkin diagrams of types  $A_{n-1}$ and $B_n$.}
\end{center}
\end{figure}

For $w\in W$,  the {\em length} of $w$, denoted $\ell(w)$, is the minimum length $\ell$ of any expression of $w$ as a product $s_{i_1}\cdots s_{i_\ell}$ with $s_{i_j} \in S$. These expressions of length $\ell$ are called {\em reduced} and denoted with a bold symbol ${\bf w}=s_{i_1} \cdots s_{i_\ell}$. %Moreover, we 
Denote by $\mathcal{R}(w)$ the set of all reduced expressions of $w$.
\medskip

The {\em right descent set} of $w \in W$ is
\begin{equation}\label{def:DesW}
\Des(w):=\{s \in S \mid \ell(w s)<\ell(w)\}.
\end{equation}

If ${\bf w}=s_{i_1} \cdots s_{i_\ell}$ is a reduced expression for $w$, then a reduced expression for $w^{-1}$ is given by $s_{i_\ell} \cdots s_{i_1}$. It follows that
%\begin{equation}\label{eq:inverse}
%    \FC(w)^{-1}:=\{w^{-1|\ w\in \FC(w)\}=\FC(w)}
%\end{equation}
%and
\begin{equation}\label{def:inverseDesW}
\Des^L(w):=\{s \in S \mid \ell(sw)<\ell(w)\}=\Des(w^{-1}),
\end{equation}
known as the {\em left descent set} of $w$. 
\medskip

For $J \subseteq S$, denote by $W_J$ the {\em parabolic subgroup} of $W$ generated by $J$, and by  
$$W^J:=\{w \in W \mid \Des(w)\subseteq S \setminus J\},$$
the {\em set of minimal coset representatives}, or {\em quotient}. The next result is well known, see for example~\cite[Proposition 2.4.4]{BB}.

\begin{proposition}\label{prop:quotient}
For every $J \subseteq S$ the following holds.
\begin{itemize}
\item [(i)] Every $w \in W$ has a unique factorization $w=w^J \cdot w_J$ such that $w^J \in W^J$ and $w_J \in W_J$.
\item [(ii)] For this factorization $\ell(w)=\ell(w^J)+\ell(w_J)$.
\end{itemize}
\end{proposition}

The well-known {\em Matsumoto-Tits word property} ensures that any reduced expression of $w \in W$  can be obtained from any other using only braid relations (see for instance~\cite{Hu}). The concept of full commutativity is a strengthening of this property.

\begin{definition}
\label{defi:FC}
An element $w$ is \emph{fully commutative} (FC) if any reduced expression for $w$ can be obtained from any other one by using only commutation relations.
\end{definition}

The following characterization of FC elements, originally due to Stembridge, is particularly useful for checking  whether a given element is FC.  
\begin{prop} %[Stembridge 
\cite[Prop. 2.1]{ST1}
\label{prop:caracterisation_fullycom}
An element $w\in W$ is fully commutative if and only if for all $s,t$ such that $3\leq m_{st}<\infty$, there is no reduced expression for $w$ that contains the factor $\underbrace{sts\cdots}_{m_{st}}$.
\end{prop}

We let $S^*$ be the free monoid generated by $S$. Define the following equivalence relation on $S^*$: two words $a,b\in S^*$ are equivalent if $b$ can be obtained from $a$ by a finite sequence of commutation relations. The equivalence classes of this relation   are usually called {\em commutation classes}. 
By definition, for a FC element $w$, the set $\mathcal{R}(w)$ of {\em reduced expressions} of $w$ forms a single commutation class ; we will see in Section~\ref{sec:heaps} that the concept of heap, as originally defined by Viennot~\cite{ViennotHeaps}, helps to capture the notion of full commutativity.

%%%%%%%%%%%%%%%%%%%%%%%%%%%%%%%%%%%%%%%%%%%%%%%%%
\subsection{The hyperoctahedral group}\label{sec:Bn} 
%%%%%%%%%%%%%%%%%%%%%%%%%%%%%%%%%%%%%%%%%%%%%%%%%
\smallskip

In this section we fix a positive integer $n \geq 2$. Recall that the Coxeter group of type $A_{n-1}$ is isomorphic to the symmetric group $S_n$, i.e. the group of bijections from the set $[n]$ onto itself. Similarly, the Coxeter group of type $B$ and rank $n$ may be realized as the {\em group of signed permutations} $B_n$, that is  the group of all bijections $w$ of the set $[\pm n]:=\{\pm1,\pm 2,\dots, \pm n\}$ onto itself such that
$$
w(-i) = -w(i)
$$
for every $1\le i \le n$, with composition as the group operation.
This group is also known as  the {\em hyperoctahedral
group} of rank $n$. We identify $S_{n}$ as a subgroup of $B_{n}$,
and $B_{n}$ as a subgroup of $S_{[\pm n]}$ in the natural ways.

If $w \in B_{n}$, we write $w = [w_{1}, \ldots, w_{n}]$ to
mean that $w(i)=w_{i}$ for $1\le i\le n$, and we set 
$$\Neg(w):=\{i \in [n] \mid w_i<0\},$$
the {\em negative set} of $w$.

As Coxeter generating set for $B_n$ we take $S:=\{s_i \mid \ 0\le i<n\}$, where
$s_0:=[-1,2,3,4,\dots,n]$ and, for $1\le i<n$, $s_i:=[1,\ldots,i-1,i+1,i,i+2,\ldots,n]$, see Figure~\ref{AB-types}, right.

It is well known, ~\cite[Proposition 8.1.2]{BB}, that by letting $w_0:=0$, the right descent set, defined in~\eqref{def:DesW}, is identified for a signed permutation $w$ to
%in bijection with 
the set of indices 
\begin{equation}\label{eq:DesB}
\Des_B(w) :=\{0\le i\le n-1 \mid  w_i>w_{i+1}\}. 
\end{equation}
Similarly, for a permutation $\pi \in S_n$ the right descent set is identified with the set 
\begin{equation}\label{eq:DesA}
\Des(\pi):=\{1\leq i\leq n-1 \mid \pi_i >\pi_{i+1}\}.
\end{equation}
For %a signed permutation 
$w\in B_n$, let $\ldes(w)$ be the {\em maximal descent} in $\Des_B(w$); if the descent set is empty we set $\ldes(w):=0$.
%(see Section~\ref{sec:background} for more details).

\subsection{Chow's quasi-symmetric functions, domino tableaux and bi-tableaux}\label{sec:2.3}

For an infinite set of formal variables $x_1,x_2,\dots$, the {\it Gessel fundamental quasi-symmetric function} indexed by a set $J\subseteq [n-1]$ is defined as
\[
F_J(x_1,x_2,\ldots):=\sum\limits_{0<i_1\le i_2\le \cdots\le i_n \atop j\in J\Rightarrow i_j<i_{j+1}} x_{i_1}\cdots x_{i_n}.
\]
The descent set of a standard Young tableau $T$ of size $n$ is defined as
\[
\Des(T):=\{0<i<n \mid i+1\ \text{is in a lower row than}\ i\}.
\]
The above quasi-symmetric functions are related to the classical symmetric Schur functions by the following result.

\begin{theorem}\cite[Theorem 7.19.7]{EC2}\label{thm:schur}
For every partition $\lambda\vdash n$,
\[
\sum\limits_{T\in \SYT(\lambda)} F_{\Des(T)} = s_\lambda,
\]
where $s_\lambda=s_\lambda(x_1,x_2,\ldots)$ is the Schur function indexed by $\lambda$.
\end{theorem}

This mechanism has been  extended to the framework of the hyperoctahedral group $B_n$ in two different ways. One approach
was %taken by Adin et al ~\cite{AAER} and is based on Poirier's 
introduced by Poirier~\cite{Po}, determining the signed quasi-symmetric functions, see also~\cite{Petersen, AAER}. 
%Here we rather 
In this paper we
follow the second approach, based on the work of Chow ~\cite{Chow},
which is relevant to our purposes. After defining Chow's type $B$ quasi-symmetric functions, we will recall the necessary background on tableaux used in this theory. 

\medskip

%\todo{give the precise references below}

\begin{definition}\label{def:Chow_fundamental}
For an infinite set of formal variables $x_0,x_1,x_2,\dots$, {\em Chow's type $B$ fundamental quasi-symmetric function} indexed by $J\subseteq \{0\}\cup [n-1]$ is defined as
\[
F^B_J(x_0,x_1,\ldots):=\sum\limits_{0 \leq i_1\le i_2\le \cdots\le i_n\atop j\in J\Rightarrow i_j<i_{j+1}} x_{i_1}\cdots x_{i_n},
\]
where $i_0:=0$.
\end{definition}

\begin{example}
For $n=3$, we have 
$F^B_{\{1,2\}}=\displaystyle{\sum_{0\le i<j<k} x_ix_jx_k}$, and $F^B_{\{0,2\}}=\displaystyle{\sum_{1\le i \le j < k} x_ix_jx_k}$.
\end{example}

The Chow type $B$ fundamental quasi-symmetric functions are intimately related to domino tableaux.

\begin{definition} Let $\lambda \vdash 2n$ be a partition.  
\begin{itemize} 
\item[1.]  
A {\em standard domino
tableau} of shape $\lambda$ consists of a tiling of the Young diagram
of $\lambda$ by dominoes which are labelled  
by %non-negative integers.  If the entries are 
$1,2,\ldots,n$, 
%If the 
such that the entries are strictly increasing along rows when read from left to right and along columns
when read from top to bottom.
%, the domino tableau is {\em standard}.
%such that each letter fills $2m$ cells, $0\le m\le n$,
%A {\em standard domino tableau} of shape $\lambda\vdash 2n$ is a filling of the cells of the Young diagram of $\lambda$ by the letters $1,\ldots,n$ such that
%each letter $1\le i\le n$ fills exactly two adjacent cells, and 
%%entries weakly increase along rows from left to right and along columns 
%%upside down.
%for each $1 \le k \le n$, the union of the cells filled by the letters $\le k$ forms a Young diagram of ordinary shape. 
Denote by $P^0(n)$ the set of partitions $\lambda\vdash 2n$ that can be filled by dominoes and by $\DSYT(\lambda)$
the set of 
standard domino tableaux of shape $\lambda$.

\item[2.] If the dominoes are labelled by non-negative integers, and entries  
%of the domino tableau are 
%non-negative integers, which 
are 
weakly increasing along the rows %when read from left to right 
and strictly increasing along the columns,  %when read from top to bottom,
the domino tableau is {\em semi-standard}.
Denote 
 by $\SSDT(\lambda)$ the set of semi-standard domino tableaux of shape $\lambda$,
 which satisfy  the following additional condition: if the upper leftmost domino is vertical then it cannot be labelled by $0$.
The {\em content} of a semi-standard domino tableau $\sT$ is defined to be $w(\sT)=(\mu_0,\mu_1,\dots)$ where for each $i$, $\mu_i$ is the number of appearances of the number $i$ in $\sT$. 
\end{itemize}
\end{definition}

%\begin{definition}\label{semi-standard domino}
%For a partition $\lambda \vdash 2n$,  a {\it semi-standard domino tableau} of shape $\lambda$ is a filling of the Young diagram of $\lambda$ with horizontal and vertical dominoes labelled with integers in the set $\{0,1,\dots,\}$ in such a way that the numbers are non-decreasing along the rows, strictly increasing along the columns, with the additional condition that if the leftmost domino is vertical then it cannot be labelled by $0$. Denote by $\SSDT(\lambda)$ the set of semi-standard domino tableaux of shape $\lambda$. 
%The weight of a semi-standard domino tableau $\sT$ is defined to be $w(\sT)=(\mu_0,\mu_1,\dots)$ where for each $i$, $\mu_i$ is the number of appearances of the number $i$ in $\sT$. 
%\end{definition}

%\todo{YR: Above 2 definitions are not coherenet. Second, do we use 
%$\SSDT$? Why should we define them?}

%\begin{example}
%\todo{an example here?}
%\end{example}

Domino tableaux will be denoted in serif mode (for instance $\sT$) to distinguish them from classical tableaux and bi-tableaux (for instance $T$). 

Generating functions for domino tableaux, or {\it domino functions} are well studied, see e.g.~\cite{KLLT}. Here we use a modified version due to Mayorova and Vassilieva~\cite{MV2}.

%Now we can define combinatorially the modified domino function, in a similar  way as Schur functions.

\begin{definition}\label{def:domino_function}
Let $\lambda \in P^0(n)$. The {\em domino function} of $\lambda$ is the generating function
$$\mathcal{G}_{\lambda}({\bf x}):=\sum\limits_{\sT \in \SSDT(\lambda)}{\bf x}^{w(\sT)},$$
where ${\bf x}^{w(\sT)}:=\prod_{i\geq 0} x_i^{\mu_i}$.
\end{definition}

The {\em standard descent set} of a standard  domino tableau $\sT$ consists of all letters $1\le i< n$, such that the northeast  cell filled by $i+1$ is in a lower row than the northeast cell filled by $i$. 
Denote the letter in the $(i,j)$ cell of $\sT$ by $\sT_{i,j}$.
The {\em type $B$ descent set} of a standard domino tableau $\sT$ of size $n$ is defined as
\[
\Des_B(\sT):=\begin{cases}\Des(\sT)\sqcup\{0\} & \mbox{if} \; \sT_{2,1}=1,\\
\Des(\sT) & \mbox{if} \; \sT_{1,2}=1.
\end{cases}
\]

\begin{example}
Here are two domino tableaux $\sT,\sP\in \DSYT(4,4,2)$
\[
\sT = \ytableaushort{1133, 2455, 24}\ 
 \  \ \ \ , \ \ \ \ \ \ 
\sP=  \ytableaushort{1225, 1335, 44}\ 
\]
with descent sets $\Des(\sT)=\Des_B(\sT)=\{1,3\}$ and $\Des(\sP)=\{2,3\}\subsetneq \Des_B(\sP)=\{0,2,3\}.$ 
\end{example}

%\todo{give the reference below}

%In ~\cite{MV2} the authors show that the domino function defined above can be written as a sum of Chow's type $B$ fundamental quasi-symmetric functions, similarly \bj{as} Theorem~\ref{thm:schur} for type $A$.

The following type $B$ analogue of Theorem~\ref{thm:schur} holds.

\begin{prop}\cite[Prop. 3.9]{MV2}\label{prop:domino-chow}
For every partition $\lambda \in P^0(n)$, 
$$\sum\limits_{\sT\in \DSYT(\lambda)}F^B_{\Des_B(\sT)}=\mathcal{G}_{\lambda}.$$
\end{prop}

\medskip

Recall the hook formula for the number of domino tableaux of given shape. Denote $f^\lambda:=\#\SYT(\lambda)$ and
$f_2^\lambda:=\#\DSYT(\lambda)$. Let $[\lambda]$ be the Young diagram of shape $\lambda$ and $h_{i,j}$ be the {\em hook length} of the cell $(i,j)\in [\lambda]$, that is the number of cells in the $i$-th row and $j$-th column minus $i+j$.

\smallskip

\begin{theorem}~\cite[Theorem 14.9.18]{AR}
For every partition $\lambda \in P^0(n)$ 
\begin{equation}\label{domino-hook}
f_2^\lambda=\frac{n!}{\prod\limits_{c\in [\lambda]\atop  h_c \; \text{is even}}\frac{h_c}{2}}. 
\end{equation}
\end{theorem}

\begin{corollary}\label{cor:degree} For every $n\ge 0$, we have:
\begin{eqnarray*}
\sum_{k=0}^n \left(f_2^{(2n-k,k)}\right)^2&=&\binom{2n}{n};\\
\sum_{k=1}^{\lfloor n/2\rfloor} \left(f_2^{(2n-2k,2k-1,1)}\right)^2&=&\frac{1}{n+1}\binom{2n}{n}-1.
\end{eqnarray*}
\end{corollary}

\begin{proof}
From \eqref{domino-hook}, by considering the parity of $k$, we derive
$$f_2^{(2n-k,k)}=\binom{n}{\lfloor{k/2\rfloor}},$$
and the first formula follows by using 
$$\sum_{k=0}^n \binom{n}{\lfloor{k/2\rfloor}}^2=\sum_{k=0}^n \binom{n}{k}^2.$$
Comparison of~\eqref{domino-hook} with the hook formula for SYT~\cite[Theorem 14.5.3]{AR} yields 
$$f_2^{(2n-2k,2k-1,1)}=f^{(n-k,k)},$$
for every $n\ge 2 k\ge 2$. Recall that the {\em Catalan number} $C_n:=\frac{1}{n+1}\binom{2n}{n}$ counts the set of $321$-avoiding permutations in $S_n$, which in turn corresponds via RSK to the set of pairs of tableaux of the same shape $\lambda \vdash n$ containing at most $2$ rows, see \cite[Corollary 7.23.12]{EC2}.  
Therefore $\sum_{k=0}^{\lfloor n/2\rfloor} (f^{(n-k,k)})^2=C_n$, and the second equation follows.
\end{proof}

\bigskip

A family of skew shapes which plays an important role in the type $B$ theory is the following.
A {\it bi-shape} $(\lambda^-,\lambda^+)\vdash n$ is a pair of partitions of total size $n$.
We draw the bi-shape $(\lambda^-,\lambda^+)$ so that the southwest corner of the
component of shape $\lambda^+$ is identified with the northeast corner of
the component of shape $\lambda^-$.
Denote {\em the set of standard Young tableaux of bi-shape} $(\lambda^-,\lambda^+)$ by 
$\BSYT(\lambda^-,\lambda^+)$.
For $T\in \BSYT(\lambda^-,\lambda^+)$,
let $T_{\lambda^-}$, $T_{\lambda^+}$ be the components of $T$ of shape $\lambda^-$ and $\lambda^+$ respectively.

The {\em standard descent set} of a standard Young bi-tableau $T$ of size $n$ is defined as
\begin{equation}\label{eq:bi_st}
\Des(T):=\{0<i<n \mid i+1\ \text{is in a lower row than}\ i\},
\end{equation}
while the {\em type $B$ descent set} of a standard Young bi-tableau $T$ of size $n$ is defined as
\begin{equation}\label{eq:bi_desB}
\Des_B(T):=\begin{cases}\Des(T)\sqcup\{0\} & \mbox{if } 1 \in T_{\lambda^-},\\
\Des(T) & \mbox{if } 1 \in T_{\lambda^+}.
\end{cases}
\end{equation}

\begin{example}
Here are two standard bi-tableau of shape $((2),(2,1))$ 
with descent sets $\Des(T)=\Des_B(T)=\{1,3\}$ and $\Des(P)=\{2,3\}\subsetneq \Des_B(P)=\{0,2,3\}$:
\[
T = \ytableaushort{\none\none 13,\none\none 2, 45}\ \ 
 \  \ \ \ , \ \ \ \ \ \ 
P=  \ytableaushort{\none\none 25,\none\none 3, 14}\ \ \ .  \]

\end{example}

\bigskip

\subsection{From domino tableaux to bi-tableaux}\label{sec:Littlewood}

There exists a well-known  bijection from (semi)-standard domino tableaux
of shape $\lambda\in P^0(n)$, to bi-(semi)-standard-Young tableaux of corresponding bi-shape $(\lambda^-,\lambda^+)\vdash n $, due to Carr\'e--Leclerc, see~\cite[Algorithm 6.1]{Carre-Leclerc}.
%We briefly recall its construction in the case of standard domino tableaux, using the reformulation in \cite{MV2}. 
The bijection associates each 
(semi)-standard domino tableau $\sT$ of shape $\lambda$
%$\sT \in \DSYT(\lambda)$ 
with a pair $(T^-,T^+)$ of (semi)-standard Young tableaux of shapes $(\lambda^-,\lambda^+)$. The tableaux $T^-$ and $T^+$ are constructed as follows : assign to each (single) box of $\sT$ a sign $-$ or $+$ such that the upper leftmost box is assigned  a $-$ and two adjacent boxes have opposite signs. The  component $T^-$ (resp. $T^+)$ is then obtained from the sub-tableau of $\sT$ composed of the dominoes with upper rightmost box filled with $-$ (resp. $+$). 

Note that the resulting shapes $(\lambda^-,\lambda^+)$ only depend on the shape $\lambda$, and they correspond to the {\em 2-quotient} obtained from $\lambda$ by the {\em Littlewood decomposition} (see~\cite[page 468]{EC2} or~\cite[Chapter 14.9]{AR}); in this case the $2$-core of $\lambda$ is empty. 
We will denote this particular case of the Littlewood decomposition by $\psi(\lambda):=(\lambda^-,\lambda^+)$.

\begin{example}
Let $\sT$ be the following domino tableau: 

\[
\begin{tikzpicture}[cap=round,line width=0.2pt,scale=0.5]
%First tableau
\draw[-](0,0)--(4,0);
\draw[-](0,-1)--(4,-1);
\draw[-](0,-2)--(2,-2);
\draw[-](0,-3)--(2,-3);
\draw[-](0,0)--(0,-3);
\draw[-](1,0)--(1,-3);
\draw[-](2,0)--(2,-3);
\draw[-](3,0)--(3,-1);
\draw[-](4,0)--(4,-1);
\draw[-](2,0)--(2,-3);

\node(1a) at (0.4,-0.4) {$1$};
\node (1b) at (0.4,-1.4) {$1$};
\node(2a) at (1.4,-0.4) {$2$};
\node (2b) at (1.4,-1.4) {$2$};
\node(3a) at (0.4,-2.4) {$3$};
\node (3b) at (1.4,-2.4) {$3$};
\node(4a) at (2.4,-0.4) {$4$};
\node (4b) at (3.4,-0.4) {$4$};
\end{tikzpicture}
\]

We assign the $\pm$ signs to get the following:

\[
\begin{tikzpicture}[cap=round,line width=0.2pt,scale=0.5]
\draw[-](0,0)--(4,0);
\draw[-](0,-1)--(4,-1);
\draw[-](0,-2)--(2,-2);
\draw[-](0,-3)--(2,-3);
\draw[-](0,0)--(0,-3);
\draw[-](1,0)--(1,-3);
\draw[-](2,0)--(2,-3);
\draw[-](3,0)--(3,-1);
\draw[-](4,0)--(4,-1);
\draw[-](2,0)--(2,-3);
\node(1a) at (0.4,-0.4) {$1$};
\node(1am) at (0.7,-0.3) {$-$};
\node (1b) at (0.4,-1.4) {$1$};
\node(1bp) at (0.8,-1.3) {$+$};

\node(2a) at (1.4,-0.4) {$2$};
\node(2ap) at (1.8,-0.4) {$+$};

\node (2b) at (1.4,-1.4) {$2$};
\node(2bm) at (1.8,-1.4) {$-$};

\node(3a) at (0.4,-2.4) {$3$};
\node(3am) at (0.8,-2.3) {$-$};

\node (3b) at (1.4,-2.4) {$3$};
\node(3bp) at (1.8,-2.3) {$+$};
\node(4a) at (2.4,-0.4) {$4$};
\node(4am) at (2.8,-0.4) {$-$};

\node (4b) at (3.4,-0.4) {$4$};
\node(4ap) at (3.8,-0.4) {$+$};
\end{tikzpicture}
\]

According to the algorithm, the corresponding standard Young bi-tableau is :

\[
\begin{ytableau}
\none & {2} & {4}\\
\none & {3} \\
  {1} 
\end{ytableau}
\]
\end{example}

\smallskip

%\red{
%In the case of semi-standard domino tableaux, the algorithm is very similar. 
Consider the case of semi-standard domino tableaux. 
Note that the condition %stipulated in Definition \ref{semi-standard domino} 
that $0$ must not occupy a vertical upper leftmost domino implies that in the semi-standard bi-tableau associated with a domino semi-standard tableau, $0$ will not appear in the lower component.  
%By 
Now, the  Carr\'e-Leclerc bijection
%the algorithm described above 
from semi-standard domino tableaux to semi-standard bi-tableaux
%described above 
is content-preserving,
implying that 
\[
\sum\limits_{\sT\in \SSDT(\lambda)}{\bf x}^{w(\sT)}=s_{\lambda^-}(x_1,x_2,\dots)s_{\lambda^+}(x_0,x_1,\dots).
\]
Thus, by Proposition~\ref{prop:domino-chow}, we derive the following.

\begin{proposition}\cite[Prop. 3.13]{MV2}\label{thm:schurB}
For every $\lambda\in P^0(n)$, %$\lambda\vdash 2n$  with $f_2^\lambda\ne 0$ %, we have 
\[
\sum\limits_{\sT \in \DSYT(\lambda)} F^B_{\Des_B(\sT)}(x_0,x_1,\ldots) = s_{\lambda^-}(x_1,x_2,\ldots)\ s_{\lambda^{+}}(x_0,x_1,x_2,\ldots).
\]
\end{proposition}

Comparing Proposition~\ref{thm:schurB} with~\cite[Prop. 4.2]{AAER}, Mayorova and Vassilieva deduce the following result.

\begin{lemma}\cite[Lemma 3]{MV4}\label{lemma:MV}
For every $\lambda\in P^0(n)$  %with $f_2^\lambda\ne 0$ 
there exists an implicit $\Des_B$-preserving bijection from 
the set of standard domino tableaux of shape $\lambda\vdash 2n$ to the set of standard
bi-tableaux of bi-shape $(\lambda^-,\lambda^+)$.
\end{lemma}

%\bj{\todo{We cannot cite the preprint so we should give a proof. It seems that one step is missing. Moreover Poirier quasi symmetric function are needed but they are defined only later on in Section 8 }}

%The following remark will be used in the proofs of below statements.

\medskip

In particular, we have the following remark that will be used in Section~\ref{sec:proofMain}.

\begin{remark}\label{rem:bijections}
%Note that f
The 2-quotient of the domino shape $\lambda=(2n-k,k)$
%,  $(\lambda^{-},\lambda^{+})$, obtained by the Littlewood correspondence
%\todo{YR: maybe Carr\'e-Leclerc's bijection?}
is the bi-shape $((k/2),(n-k/2))$ if $k$ is even, and $((n-(k-1)/2),((k-1)/2))$ if $k$ is odd.  
The 2-quotient of the domino shape $\lambda=(2n-2k,2k-1,1)$ is $(\emptyset,(n-k,k))$. 
By Lemma~\ref{lemma:MV}, 
there exist $\Des_B$-preserving maps
%\begin{corollary}\label{thm:V} 
\begin{itemize}
    \item[(1)] %There exists a $\Des_B$-preserving map 
from $\DSYT(2n-2k,2k)$ to  
$\BSYT((k), (n-k))$ for $0\leq k \leq \lfloor n/2 \rfloor$; 
%similarly, there exists a $\Des_B$-preserving map 
\item[(2)]from  $\DSYT(2n-2k-1,2k+1)$ to $\BSYT((n-k),(k))$, for $1\leq k \leq \lfloor (n-1)/2 \rfloor$; 
% Also, 
\item[(3)] %    There     there exists a $\Des_B$-preserving map 
from  $\DSYT(2n-2k,2k-1,1)$ to $\BSYT(\emptyset,(n-k,k)))$, for $1\leq k \leq \lfloor n/2 \rfloor$.
\end{itemize}
%\end{corollary}
Note that the Carr\'e-Leclerc bijection is not $\Des_B$-preserving in general,
but it is for domino shapes of the form  $\lambda=(2n-2k,2k-1,1)$.
\end{remark}

%%%%%%%%%%%%%%%%%%%%%%%%%%%%%%%%%%%%%%%%%%%%%%%%%
\section{Heaps and full commutativity}~\label{sec:heaps}
%%%%%%%%%%%%%%%%%%%%%%%%%%%%%%%%%%%%%%%%%%%%%%%%%

%%%%%%%%%%%%%%%%%%%%%%%%%%%%%%%%%%%%%%%%%%%%%%%%%
\subsection{Types A and B}~\label{sec:heapAB}
%%%%%%%%%%%%%%%%%%%%%%%%%%%%%%%%%%%%%%%%%%%%%%%%%

We briefly describe a way to define the above mentioned heaps and their relation with full commutativity, for more details see for instance~\cite{BJN} and the references cited there.

Let $(W,S)$ be a Coxeter system, and fix a word $\mathbf{w}=s_{a_1}\cdots s_{a_l}$ in $S^*$. Define a partial ordering $\prec$ on the index set $\{1,\ldots, l\}$ as follows: set $i\prec j$ if $i<j$ and $s_{a_i}$, $s_{a_j}$ do not commute, and extend by transitivity. We denote by $\H({\mathbf{w}})$ this poset together with a labeling map $\epsilon:i\mapsto s_{a_i}$. Heaps are well-defined up to commutation classes~\cite{ViennotHeaps}, that is, if ${\bf w}$ and ${\bf w'}$ are two reduced expressions for $w \in W$, that are in the same commutation class, then the corresponding labeled heaps are isomorphic. Therefore, when $w$ is FC  we can define $\H(w):=\H(\mathbf{w})$, where ${\bf w}$ is any reduced expression for $w$. Another important feature in heaps theory is that the linear extensions of  $\H(w)$ are in bjiection with the reduced expressions of $w$, see \cite[Proposition 2.2]{ST1}.

\begin{example}
Consider $w=[4,1,5,2,3] \in \FC(A_4)=\FC(S_5)$. Its heap is represented in Figure~\ref{examplesA1}, left. In the Hasse diagram of $\H(w)$, elements with the same labels are drawn in the same column. We recall that each vertex is labeled by the corresponding generator, but we do not write those labels for visibility reasons.  

Its set of reduced expressions $\mathcal{R}(w)=\{s_3s_2s_1s_4s_3, s_3s_2s_4s_1s_3, s_3s_2s_4s_3s_1, s_3s_4s_2s_3s_1, s_3s_4s_2s_1s_3 \}$ is obtained by listing the labels of each linear extension of $\H(w)$.  
\end{example}

\begin{figure}[!ht]
\begin{center}
\includegraphics[scale=0.7]{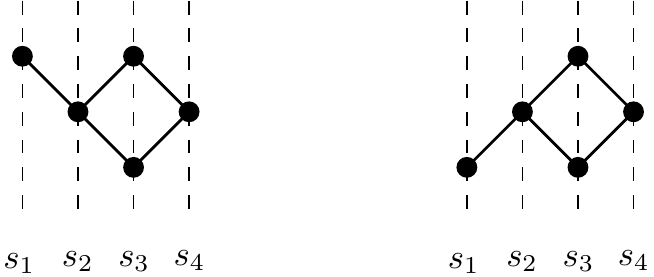}
\caption{Two FC heaps of type $A_4$.}\label{examplesA1}  
\end{center}
\end{figure}

%\red{
Given a heap $H=\H(w)$ for $w \in \FC(W)$ and a subset $I\subset S$, we denote by $H_{I}$ the sub-poset induced by all elements of $H$ with labels in $I$.
 %}
\begin{comment}
Given any subset $I\subset S$, we will denote by $H_{I}$ the sub-poset induced by all elements of $H$ with labels in $I$.
\end{comment}

\begin{definition}
\label{defi:alternating}
Let $(W,S)$ be a Coxeter system, $w \in \FC(W)$, and $H:=\H(w)$. We say that $H$ is {\em alternating} if for each non commuting generators $s,t$ in $S$, the chain $H_{\{s,t\}}$ has alternating labels $s$ and $t$ from bottom to top.
\end{definition}

Note that if $\H(w)$ is alternating, then any reduced expression $\mathbf{w}$ of $w$ is \emph{alternating} in the sense that for each  non commuting generators  $s,t \in S$, the occurrences of $s$ alternate with those of $t$ in $\mathbf{w}$. In this case we say that $w \in \FC(W)$ is alternating.
\medskip

We now recall the descriptions of FC heaps corresponding to the Dynkin diagrams of types $A_{n-1}$ and $B_n$ which were given  for instance in~\cite{BJN}. 

\begin{proposition}[Classification of FC heaps in type $A_{n-1}$]\label{heaps-typeA}
An element $w \in A_{n-1}$ is fully commutative if and only if $\H(w)$ is alternating. More precisely, in $\H(w)$, 
\begin{itemize}
\item[(a)] There is at most one occurrence of $s_1$ ({\em resp.} $s_{n-1}$);
\item[(b)] For each $i \in \{1,\ldots, n-2\}$, the elements with labels $s_i, s_{i+1}$ form an alternating chain. 
\end{itemize}
\end{proposition}

As already mentioned, such elements are in bijection with $321$-avoiding permutations in $S_n$, that are counted by the Catalan number $C_n$. In Figure~\ref{examplesA}, the heaps of three FC elements of type $A$ are depicted.

\begin{figure}[!ht]
\begin{center}
\includegraphics[scale=0.7]{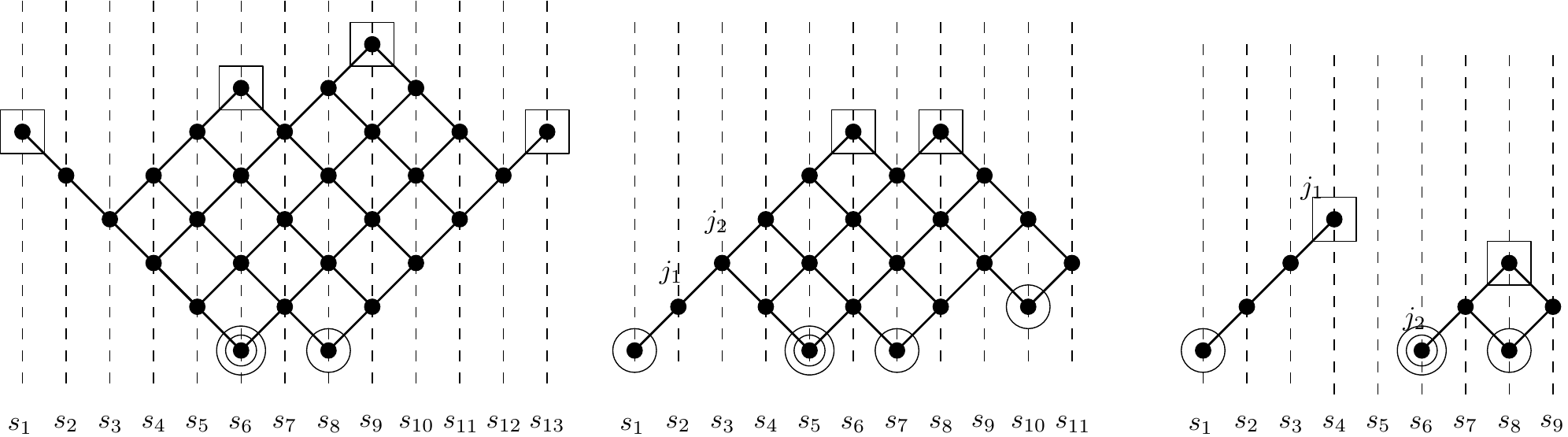}
\caption{Three alternating heaps of type $A$.}\label{examplesA}  
\end{center}
\end{figure}

Now we will need the following second family of heaps, which, in addition to alternating heaps, describes FC heaps of type $B$ (see, e.g.~\cite{BJN}).

\begin{definition}\label{defi:LP}
A {\em left-peak}, associated with the Dynkin diagram of type $B_n$,  is a heap  such that there exists a unique $j \in\{1,\dots,n-1\}$ satisfying:
\begin{enumerate}
\item[(a)]  $\H_{\{s_0, \ldots, s_j\}}=\H(s_j\cdots s_{1}s_0 s_{1} \cdots s_j)$;
\item[(b)] $\H_{\{s_{j}, s_{j+1}\}}=s_j s_j$ or $s_{j+1}s_js_js_{j+1}$ for $j<n-1$, and $s_{n-1}s_{n-1}$ for $j=n-1$; 
\item[(c)]  $\H_{\{\hat{s}_{j},s_{j+1}\ldots,s_{n-1} \}}$ is alternating, where $\hat{s}_j$ means that one occurrence of $s_j$ is deleted.
\end{enumerate}
\end{definition}

An element $w \in \FC(B_n)$ for which $\H(w)$ is a left-peak will also be called a {\it left-peak element}.  
\begin{comment}
As for alternating FC elements, we will call left peak elements in $\FC(B_n)$ any FC element $w$ such that $H_w$ is a left peak.
\end{comment}

\begin{example}
In Figure~\ref{examplesB-uncolored}, left, there is an example of an alternating heap of type $B$: note that in contrast to the type $A$ case (having at most one vertex labeled $s_1$), it can have any finite number (between $0$ and $n$) of vertices labeled $s_0$. 
In the left-peak of Figure~\ref{examplesB-uncolored}, right, we have $j=2$.
\end{example}

\begin{figure}[!ht]
\begin{center}
\includegraphics[scale=0.7]{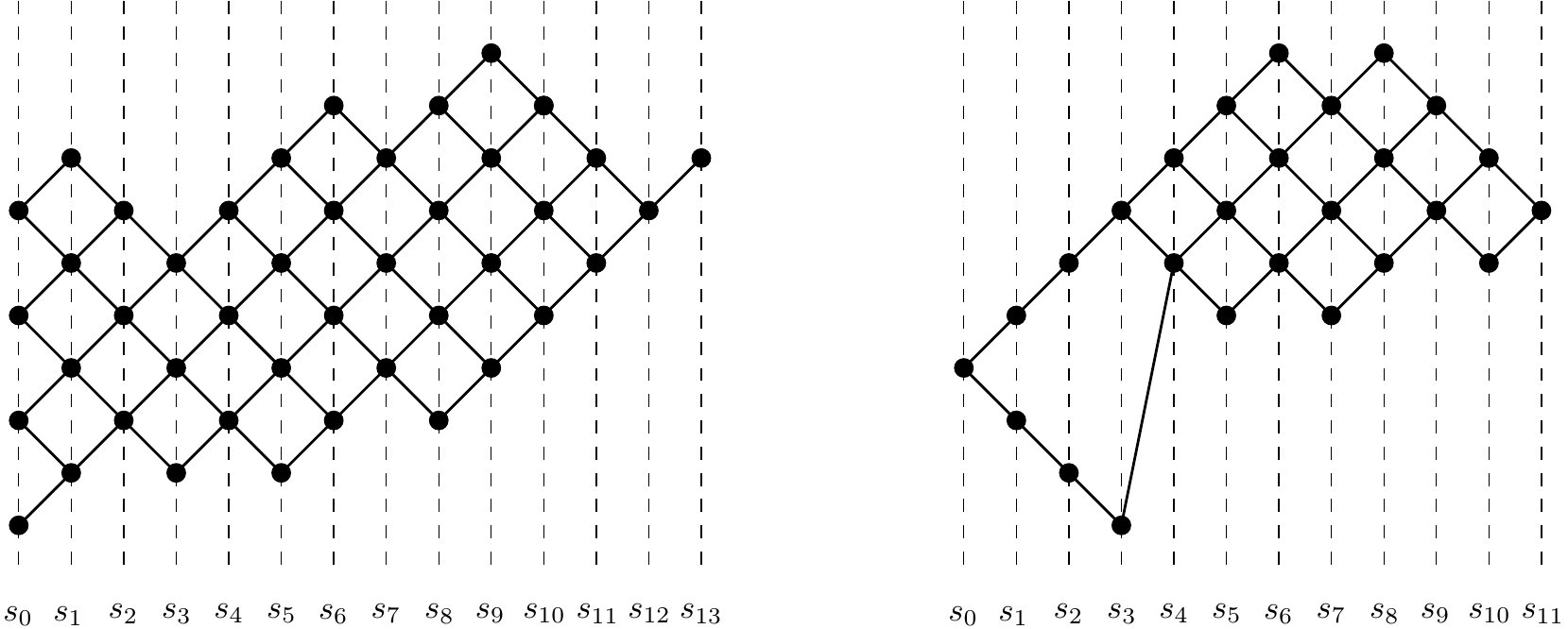}
\caption{Left: an alternating heap of type $B_{14}$. Right: a left-peak of type $B_{12}$}\label{examplesB-uncolored}  
\end{center}
\end{figure}

From  \cite[Theorem 3.10 and Section~4.4]{BJN}, we have the following result.

\begin{proposition}[Classification of FC heaps in type $B_n$]\label{heaps-typeB}
An element $w \in B_{n}$ is fully commutative if and only if $\H(w)$ is either an alternating heap or  a left-peak.  
\end{proposition}

\begin{remark}\label{number of alternatings and LP}
The set of alternating FC elements in $B_n$ coincides with the set of FC top elements defined in Theorem~4.1 of~\cite{ST3}, while the set of FC left peak elements is exactly the set of FC bottom elements which are not top elements. Therefore by~\cite[Theorem~5.9]{ST3}, the number of alternating FC elements is $\binom{2n}{n}$, and the number of FC left peaks is $C_n-1$. Thus  the total number of FC elements in $B_n$ is $$\frac{n+2}{n+1}\binom{2n}{n}-1.$$
\end{remark}

Stembridge provided a characterization of FC elements in  $B_n$ by using pattern avoidance: as for FC$(S_n)$ these elements are 321-avoiding but they also have to avoid other patterns. %, which is more complicated than in  $S_n$.

\begin{proposition}\cite[Theorem 5.1]{ST3}\label{prop:fc1} 
%Stembridge provided also a characterization of FC elements in  $B_n$ by using pattern avoidance, which is more complicated than in  $S_n$. 
A signed permutation $w\in B_n$ is fully commutative if and only if
$w$ avoids the pattern $[-1,-2]$ and all patterns $[a, b, c]$ such that $|a| > b > c$ or $-b > |a| > c$.
\end{proposition}

%Among the forbidden patterns we only mention $[-1,-2]$ and $[3,2,1]$, that we will use later.

%It follows that if $w\in \FC(B_n)$ then $w$ avoids the pattern $[3,2,1]$, but the converse does not hold.

\medskip

%%%%%%%%%%%%%%%%%%%%%%%%%%%%%%%%%%%%%%%%%%%%%%%%%
\subsection{Reduced expressions}
%of fully commutative elements and heaps}
~\label{sec:reduced-expressions}
%%%%%%%%%%%%%%%%%%%%%%%%%%%%%%%%%%%%%%%%%%%%%%%%%
%\medskip
%
%Recall that for Coxeter systems $(W,S)$ of types $A_{n-1}$ and $B_n$,  the {\em left-descents set} of any $w\in W$ is defined by 
%$$
%\Des^L(w):=\{i \mid \ell(s_iw )<\ell(w) \;\mbox{and}\; s_i\in S\},
%$$
%where the generators in $S$ are labeled as in Figure~\ref{AB-types}. On the other hand, recall that 
%$$
%\Des(w^{-1})=\Des^L(w).
%$$
Let $w\in W$ be a FC element of type $A_{n-1}$ or $B_n$. By definition $i \in \Des(w)$ if and only if there exists a peak in $\H(w)$ labeled by $s_i$, where by a {\em peak} we mean a vertex having all its neighbors vertices below it. Moreover, $i \in \Des^L(w)$ if and only if there exists a valley in $\H(w)$ labeled by $s_i$, where by a {\em valley} we mean a vertex having all its neighbors above it. Indeed, notice that $w$ is FC if and only if $w^{-1}$ is FC. Moreover, one can see that $\H(w^{-1})$ is the dual heap of $\H(w)$, i.e. the heap of $w$ with the reverse order, that can be obtained from $\H(w)$ by a horizontal reflection. In Figure~\ref{examplesA1} a heap and its dual are depicted.

For our purposes, it will be useful to introduce a new statistic on $\FC(S_n)$.

\begin{definition}\label{def:first-valley}
For any $\pi \in \FC(S_n)$ we define the {\em first valley}, denoted $v(\pi)$, as follows :
\begin{equation}
 v(\pi):=\left\{ 
\begin{array}{ll} 
\mbox{${\rm min} \left\{\Des(\pi^{-1})\setminus \{1\}\right\}$} & \mbox{if $\Des(\pi^{-1})\setminus \{1\} \neq \emptyset $;}\\
n & \mbox{if $\Des(\pi^{-1}) \setminus \{1\} = \emptyset$.}
\end{array} \right.
\end{equation}
\end{definition}

Note that $v(\pi)=n$ if and only if  $\pi=e$ or $\Des(\pi^{-1})=\{1\}.$
Moreover, for $\pi\in \FC(S_n)$,  we have that 
\begin{equation}\label{eq:pi(1)}
1 \in \Des(\pi^{-1})\Leftrightarrow  \pi(1)=2 .
\end{equation}

In Figure~\ref{examplesA}, the descents of the three elements are surrounded by a square, the valleys by a circle, and the first valley by a double circle.

As we mentioned above, if $w$ is FC then the set of linear extensions of $\H(w)$ is in bijection with the set of reduced expressions of $w$. It will be helpful in the sequel to consider a particular reduced expression for each $w \in \FC(S_{n})\cup\FC(B_{n})$.

\begin{definition}\label{def:diagonal}
The {\em diagonal reduced expression} of $w  \in \FC(S_{n}) \cup \FC(B_{n})$ is obtained  by reading the labels of the vertices in the ``diagonals" of $\H(w)$, directed from south east to north west of $\H(w)$, starting from the leftmost diagonal. Each such diagonal contributes a factor of the form $(s_is_{i-1} \cdots s_j)$ with $0\leq j\leq i\leq n-1$.  
It is easy to see that such an expression corresponds to a linear extension of $\H(w)$.
\end{definition}

More precisely, consider first $e \neq w \in \FC(S_{n})$. Then the diagonal reduced expression for $w$ is of the form
\begin{equation}\label{diagonal-expression-A}
\mathbf{w}=(s_{v_0}s_{v_0-1} \cdots s_{j_0})(s_{v_1}s_{v_1-1} \cdots s_{j_1})\cdots (s_{v_k}s_{v_k-1}\cdots s_{j_k}),
\end{equation}
where $1 \leq v_0<v_1<\ldots<v_k\leq n-1$ and $1\leq j_0<j_1<\ldots<j_k\leq n-1$. We have $v_0=v(\pi)$  if $1 \notin \Des(\pi^{-1})$ and $v_0=1$ otherwise.
\smallskip
\begin{example}
The diagonal reduced expression for the element on the left of Figure~\ref{examplesA} is
$$(s_6\cdots s_1)(s_8\cdots s_4)(s_9\cdots s_5)(s_{10}\cdots s_6)(s_{11}\cdots s_8)(s_{12}\cdots  s_9)(s_{13}),$$
while for the element on the right it is $(s_1)(s_2)(s_3)(s_4)(s_6)(s_8s_7)(s_9s_8).$
\end{example}

In $\FC(B_n)$ there are two possibilities: 
\begin{itemize}
    \item {\bf Alternating:} 
If $w_1$ is alternating then its diagonal reduced expression takes the form
\begin{equation}\label{diagonal-expression-B}
\mathbf{w}_1 =(s_{v_0}s_{v_0-1} \cdots s_{j_0})(s_{v_1}s_{v_1-1} \cdots s_{j_1})\cdots (s_{v_k}s_{v_k-1}\cdots s_{j_k}),
\end{equation}
where $0 \leq v_0<v_1<\ldots<v_k\leq n-1$ and $0 \leq j_{0} \leq j_1 \leq \ldots \leq j_k\leq n-1$ with the condition that equality between two $j_i$'s occurs only if both are $0$.  
\item {\bf Left-peak:}
If $w_2$ is a left-peak then the diagonal reduced expression has the form
\begin{equation}\label{diagonal-expression-LP}
\mathbf{w}_2=(s_{v_0} s_{v_0-1}\cdots s_{0})(s_1)\cdots (s_{{j_1}-1})(s_{v_1}s_{v_1-1} \cdots s_{j_1})\cdots (s_{v_k}s_{v_k-1}\cdots s_{j_k}),
\end{equation}
where $0 < v_0<v_1<\ldots<v_k\leq n-1$ and $1<j_1<\ldots<j_k\leq n-1$. 
\end{itemize}
\begin{example}
The diagonal reduced expression for the element on the left of Figure~\ref{examplesB-uncolored} is
$$(s_0) (s_1s_0) (s_3s_2s_1s_0) (s_5\cdots s_0) (s_6\cdots s_1)(s_8 \cdots s_4) (s_9 \cdots s_5) (s_{10} \cdots s_6) (s_{11}\cdots s_8) (s_{12} \cdots s_9) (s_{13}) $$
while the diagonal reduced expression for the element on the right  is
$$(s_3s_2s_1s_0) (s_1) (s_2) (s_5s_4s_3)(s_7\cdots s_4)(s_8\cdots s_5) (s_{10}\cdots s_6) (s_{11}\cdots s_8).$$
\end{example}

\begin{remark}\label{one line notation of left peak}
Note that if $w$ is a left-peak then we must have $w(1)=1$. Moreover, there must be unique $i>1$ and $k>1$ such that $w(i)=-k$.
\end{remark}

\begin{remark}\label{rem:descents}
Observe that in the three above expressions~\eqref{diagonal-expression-A}--\eqref{diagonal-expression-LP}, the left descents of $w$ belong to the set $\{v_0,v_1,\ldots, v_k\}$ of initial indices of the factors exhibited in the diagonal expressions. More precisely, $v_0 \in \Des^L(w)$, and for $i\geq 1$, $v_i \in \Des^L(w)$ if and only if $v_i-v_{i-1} \geq 2$. 
\end{remark}

%%%%%%%%%%%%%%%%%%%%%%%%%%%%%%%%%%%%%%%%%%%%%%%%%
\section{Decomposition of FC($B_n$) into fibers}~\label{sec:fibers}
%%%%%%%%%%%%%%%%%%%%%%%%%%%%%%%%%%%%%%%%%%%%%%%%%

In this section we let $W=B_n$ and $J=S \setminus \{s_0\}$. 
%From Proposition~\ref{prop:quotient} it follows that 
The parabolic subgroup $(B_n)_J$ is isomorphic to 
 $S_n$ and the quotient has the form 
\begin{equation}\label{quot}\
(B_n)^J:=\{\mu \in B_n \mid \Des_B(\mu)\subseteq \{0\}\}=\{ \mu \in B_n \mid \mu(1)<\ldots<\mu(n) \}.
\end{equation}
By Proposition~\ref{prop:quotient},
every $w \in B_n$ has a unique decomposition %, we write uniquely 
\begin{equation} \label{eq:decomposition}
w=\mu \cdot \pi
\end{equation}
where $\mu\in (B_n)^J, \pi \in (B_n)_J$, and 
\begin{equation} \label{sumlength}
\ell_B(w)=\ell_B(\mu)+\ell_B(\pi).
\end{equation}

Notice that $\mu$ can be written %in a combinatorial way 
as the ascending reordering of $w$,
%From here it follows that 
and the counterpart $\pi$ is the permutation in $S_n$ which records the letters $w_1,\ldots, w_n$ in the relative standard order. 
The permutation $\pi$ is called the {\em standardization} of the signed permutation $w$. %We denote it by $\tau(w)=\pi$. 
%\end{definition}
For example,  %$\tau([1,-3,-2,4])=[3,1,2,4]$ as
$[1,-3,-2,4]=[-3,-2,1,4]\cdot [3,1,2,4]$.

\medskip

We can characterize precisely the reduced expressions of the elements in $(B_n)^J$. By defining 
$$\delta_i:=s_{i-1}\cdots s_2s_1s_0$$ for  integers $i$ such that $1\leq i\leq n$, we have
\begin{equation}\label{characterization}
(B_n)^J=\{\mu \in B_n\mid \mu=\mu_1\cdots \mu_{n}, \ \mu_i \in \{e,\delta_i\}\}.
%=\{\mu \in B_n\mid\boldsymbol{\mu} \prec (s_0)(s_1s_0)\cdots (s_{n-1}\cdots s_2s_1s_0)\}.
\end{equation}

Indeed, take an element $\mu\in (B_n)^J$. If Neg($\mu$)$=\emptyset$ then $\mu=e$, otherwise set $\Neg(\mu)=\{i_1,\dots,i_k\}$  
with %$1\leq k\leq n$ and 
$-i_1<\dots<-i_k<0$. 
We have $\mu(1)=-i_1,\dots,\mu(k)=-i_k$, while $\mu(k+1),\dots,\mu(n)$ have to be positive and in increasing order, therefore $\mu$ has a reduced expression $\boldsymbol{\mu}=\delta_{i_k}\dots\delta_{i_1}$. Conversely, one checks that any element with reduced expression $\delta_{i_k}\cdots\delta_{i_1}$ satisfies the inequalities in~\eqref{quot}. 
\smallskip

This implies the following description that will be used in Section~\ref{sec:equidistribution}.

\begin{observation}\label{obs:mu}
Every $\mu\in (B_n)^J$ is an increasing sequence of $n$ letters
from $[\pm n]$ with distinct absolute values.
Hence
$\mu^{-1}$ is a shuffle of $[-k,-k+1,\dots,-1]$ with $[k+1,k+2,\dots,n]$ for some $0\le k\le n$.
\end{observation}

As clearly no {\em long braid type factor}, that is a factor of the form $s_0s_1s_0s_1$, $s_1s_0s_1s_0$ or $s_is_{i\pm 1}s_i$, can occur in all the reduced expressions of the elements in the set~\eqref{characterization}, one deduces from Proposition~\ref{prop:caracterisation_fullycom} that 
each element in $(B_n)^J$ is FC, namely
\begin{equation}\label{eq:FC-equality}
    \FC((B_n)^J):=(B_n)^J \cap \FC(B_n) =(B_n)^J.
\end{equation} 

This also follows from Proposition~\ref{prop:fc1}, since any $\mu \in (B_n)^J$ is an increasing sequence. Moreover, note that the reduced expressions for the elements in $(B_n)^J$ given in~\eqref{characterization} are the diagonal reduced expressions from Definition~\ref{def:diagonal} of the corresponding heaps.
\smallskip

\begin{figure}[!ht]
\includegraphics[scale=0.7]{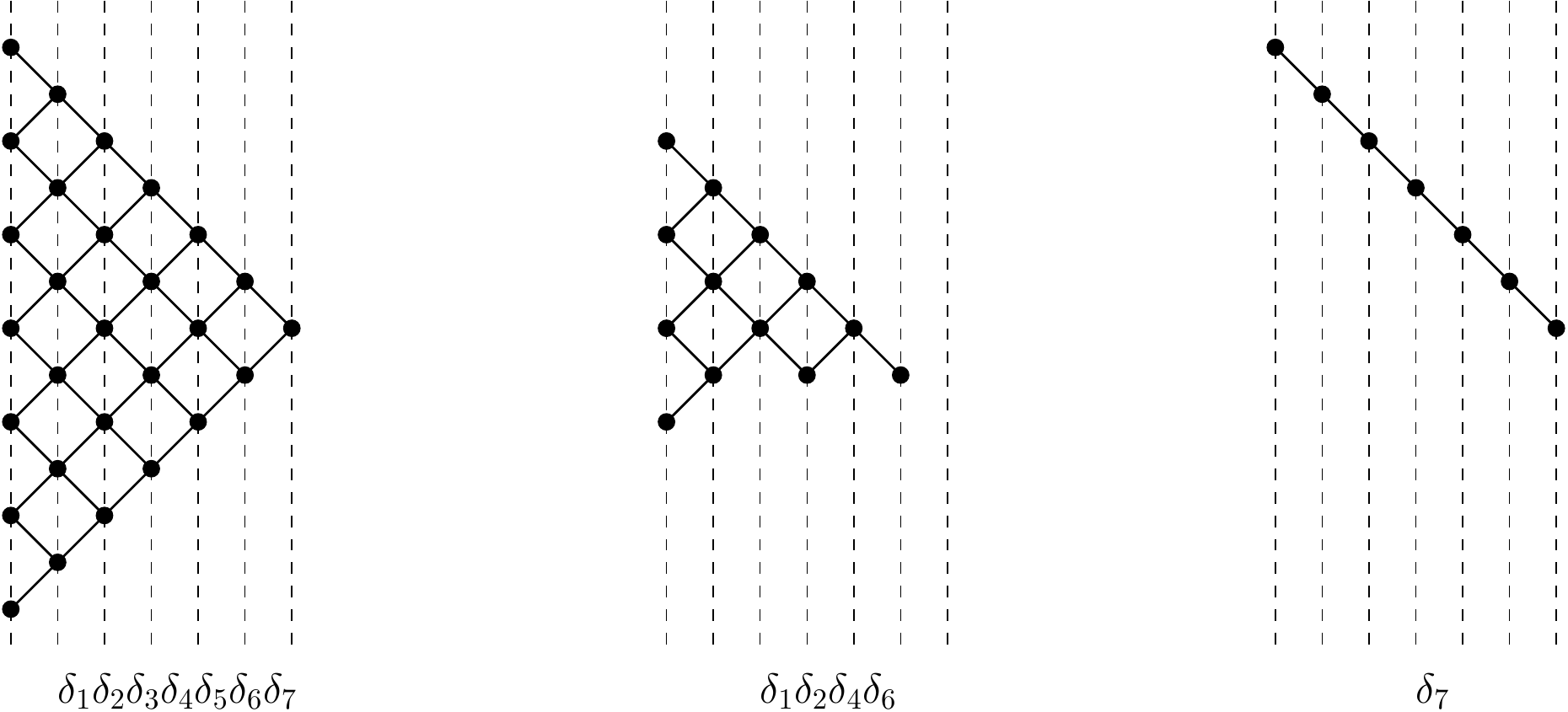}
\caption{Left: the heap of the element with maximal length of $B_7^J$, that is the element with diagonal reduced expression $\delta_1\delta_2\delta_3\delta_4\delta_5\delta_6\delta_7$. Center: the heap of the element in $B_7^J$ with diagonal reduced expression $\delta_1\delta_2\delta_4\delta_6$. Right: the heap of the element with diagonal reduced expression $\delta_7$. \label{quotientB} } 
\end{figure}

Note that the heap of any element in $(B_n)^J$ can be depicted as a sub-poset of the ``triangular'' heap corresponding to the element $\delta_1\cdots \delta_n$ having maximal length, see Figure ~\ref{quotientB}. 
\smallskip
 
Now we consider the restriction of the decomposition~\eqref{eq:decomposition} %{decomposition}  
 to FC elements, $w= \mu \cdot \pi$ %. To this aim, take 
 for $w \in \FC(B_n)$. % let %and write it as  $w= \mu \cdot \pi$. % according to~\eqref{decomposition}. 
 Then all reduced expressions of $w$ contain no braid relation, and thanks to~\eqref{sumlength} above, it implies that both $\mu$ and $\pi$ are FC. Therefore we obtain the following inclusion
\begin{equation}\label{inclusion}
\FC(B_n) \subset \FC((B_n)^J) \times \FC(S_n)=(B_n)^J \times \FC(S_n).
\end{equation}
It is easy to show that this inclusion is strict (take for instance ${\mu}=s_0s_1s_0$ and ${\pi}=s_1$).

Our next result refines the previous inclusion by exhibiting for any fixed FC permutation in $S_n$ the corresponding subset of $(B_n)^J$. 

\begin{theorem}\label{thm:fibers}
We have the following decomposition
\begin{equation}\label{disjoint_union}
\FC(B_n)=\biguplus_{\pi \in \FC(S_n)}
B_n(\pi)\cdot \pi,
\end{equation}
where  
$$B_n(\pi):=\left\{ 
\begin{array}{ll} 
\{\mu \in B_n\mid \mu=\mu_1\cdots \mu_{v(\pi)},\;\mu_i \in \{e,\delta_i\}\} & \mbox{if $1\notin\Des(\pi^{-1})$};\\
\{\mu \in B_n\mid \mu \in\{e,\delta_1,\ldots,\delta_{v(\pi)}\}\} & \mbox{if $1\in \Des(\pi^{-1})$},
\end{array} \right. $$
\smallskip
\noindent and $v(\pi)$ is the first valley of $\pi$ from Definition~\ref{def:first-valley}. 
\end{theorem}

\begin{proof}
First note that the sets on the right-hand side of~\eqref{disjoint_union} are disjoint by uniqueness of the decomposition~\eqref{eq:decomposition}.%{decomposition}.

Let us now consider an element $w \in \FC(B_n)$ and write it uniquely as $w=\mu \cdot \pi  \in (B_n)^J\times \FC(S_n)$, according to the decomposition~\eqref{eq:decomposition} and the inclusion~\eqref{inclusion}. We need to show that $\mu \in B_n(\pi)$, and to this aim we consider three cases. Along the proof we set $v:=v(\pi)$; note that by definition $v\geq 2$. 

\begin{enumerate}
\item If  $\pi=e$, we get the result by~\eqref{characterization}.

\item  If $\pi\neq e$ and $1\notin\Des(\pi^{-1})$, then there exists a reduced expression $\boldsymbol{\pi}$ of $\pi$ starting with a factor $s_{v}s_{v-1}\cdots s_{j}$, for an integer $j$ satisfying $1\leq j\leq v$. For the sake of a contradiction, assume that the rightmost factor in the reduced expression $\boldsymbol{\mu}=\mu_1\cdots\mu_n $ of $\mu$ is $\delta_i$ with $i> v$. It suffices to assume that $i=v+1$. Then me may write a reduced expression of $w$ as
$${\bf w}={\bf u}  (s_{v}s_{v-1}\cdots s_1s_0) \cdot (s_{v}s_{v-1}\cdots s_{j}){\bf \tilde{u}},$$
where ${\bf u}$ (respectively ${\bf \tilde{u}}$) is a left (respectively right) factor of $\boldsymbol{\mu}$ (respectively $\boldsymbol{\pi}$). Now between the two above occurrences of $s_v$ there is no occurrence of $s_{v+1}$, hence by applying commutation relations to ${\bf w}$ we obtain a reduced expression containing the factor $s_v s_{v-1}s_v$, a contradiction for a FC element in $B_n$ (as $v\geq 2$). An example of this case is depicted in Figure~\ref{examplesB-colored}, left.

\item If $1\in\Des(\pi^{-1})$, then we discuss two cases.  

Assume first that $\Des(\pi^{-1})=\{1\}$, which means by definition that $v=n$. Equivalently, the one-line notation of $\pi$ 
is $[2,\dots,1,\dots]$ where the elements represented by the dots are in increasing order. This means that $\pi$ has a reduced expression of the form $s_{1}s_2\cdots s_j$ for some $j\in\{1,\dots,n-1\}$. For the sake of a contradiction, suppose that no reduced expression $\boldsymbol{\mu}$ of $\mu$ belongs to $\{e,\delta_1,\ldots,\delta_{n}\}$. Then by~\eqref{characterization},  $\boldsymbol{\mu}$ contains at least two factors $\delta_i$. Let us consider the two rightmost  factors in $\boldsymbol{\mu}$, say $\delta_{i_1}$, $\delta_{i_2}$ with $1\leq i_1<i_2\leq n$. Hence 
$${\bf w}={\bf u}(s_{i_1-1}\cdots s_1s_0)(s_{i_2-1}\cdots s_2s_1s_0) \cdot (s_{1})\cdots (s_j).$$
Now the occurrence of $s_0$ in the factor $\delta_{i_1}$ commutes with all the generators in $\delta_{i_2}$ on its right up to $s_2$ included; so we can move it until the occurrence of $s_1$ in $\delta_{i_2}$, which would give a reduced expression of $w$ that contains a factor $s_0s_1s_0s_1$. This is a contradiction since $w$ is fully commutative. 
\medskip

Next assume that $\{1\} \subsetneq \Des(\pi^{-1})$. Since $\pi^{-1}$ is FC, $2 \notin \Des(\pi^{-1})$ so we must have $v\geq3$. (See an example in Figure~\ref{examplesB-colored}, right.) The diagonal reduced expression of Definition~\ref{def:diagonal}  of $\pi$ starts with the factors $(s_1)(s_2)\cdots (s_{j_1})(s_{v}s_{v-1}\cdots s_{j_2})$, where $1\leq j_1<j_2\leq v$. Note that if $j_2=j_1+1$ then $v>j_2$ (see Figure~\ref{examplesA}, center). For the sake of a contradiction, suppose that each reduced expression of $\mu$ satisfies $\boldsymbol{\mu}\not \in \{e,\delta_1,\ldots,\delta_{v}\}$. Then by~\eqref{characterization},  $\boldsymbol{\mu}$ contains either  the product of at least two different factors of the form $\delta_i$, or a single $\delta_i$ with $i>v$.
In the first situation, let us consider the two rightmost such factors in $\boldsymbol{\mu}$, say $\delta_{i_1}$, $\delta_{i_2}$ with $1\leq i_1<i_2\leq n$. Hence
$${\bf w}={\bf u}(s_{i_1 -1}\cdots s_1s_0)(s_{i_2 -1}\cdots s_2s_1s_0) \cdot (s_1)(s_2)\cdots (s_{j_1})(s_{v}s_{v-1}\cdots s_{j_2}){\bf \tilde{u}}.$$
Now by commutation relations we obtain the same contradiction as above.

In the second situation, take the rightmost factor $\delta_i$ in $\boldsymbol{\mu}$ with  $i>v$: without loss of generality, one can take $i=v+1$. We get 
$${\bf w}=(s_{v}s_{v-1}\cdots s_1s_0) \cdot (s_1)(s_2)\cdots (s_{j_1})(s_{v}s_{v-1}\cdots s_{j_2}){\bf \tilde{u}}.$$
If $j_2>j_1+1$, then $v\geq j_2>j_1+1$ (see e.g. Figure~\ref{examplesA}, right), therefore the second occurrence of $s_v$ in this expression commutes with all generators on its left up to $s_{v}s_{v-1}$, thus commutation relations would yield a factor $s_{v}s_{v-1}s_v$, a contradiction. 
If $j_2=j_1+1$, we can conclude in the same way, thanks to the condition $v>j_2>j_1$ in this case.

\end{enumerate}

\begin{figure}[!ht]
\includegraphics[scale=0.7]{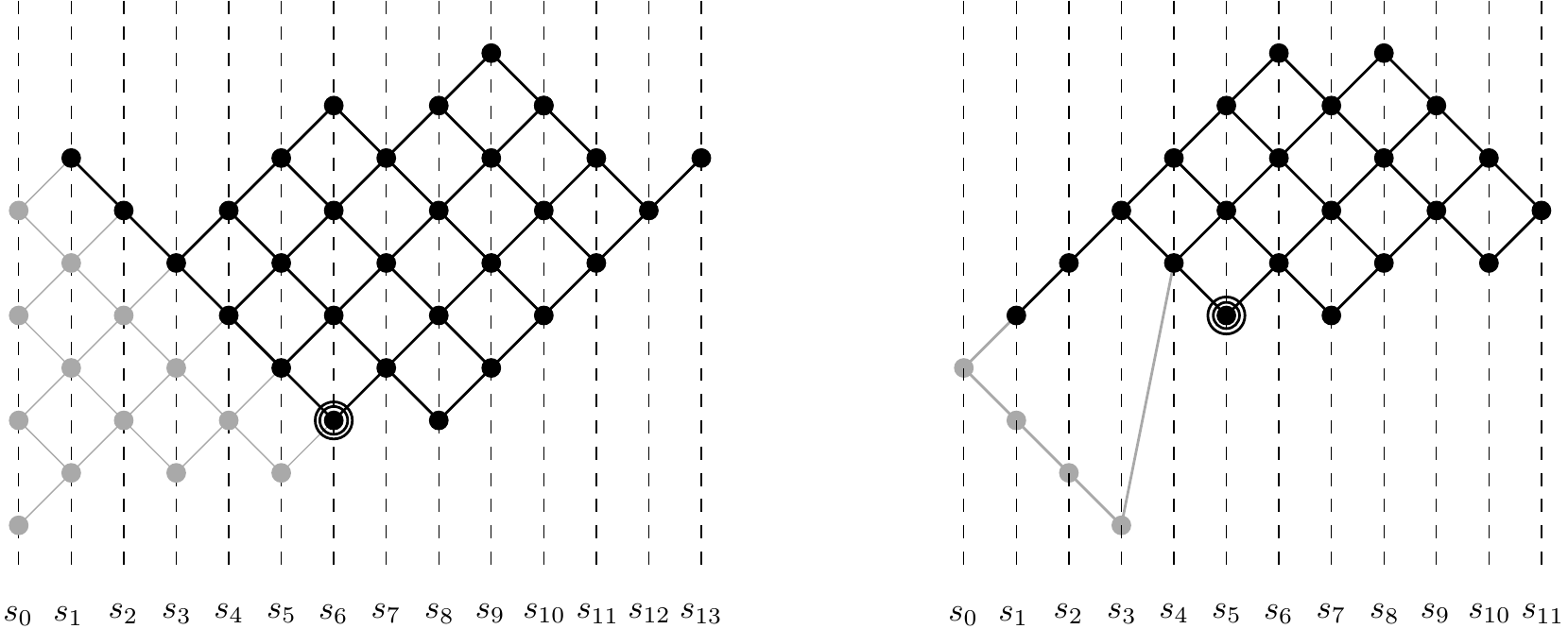}
\caption{Two products $\mu\cdot\pi$, one yielding an alternating element (left), the other giving a left-peak (right). The heap of $\mu$ is in grey. \label{examplesB-colored}}
\end{figure}

\smallskip

Let us show the opposite inclusion, by considering again three cases.
\begin{enumerate}
\item Taking $\pi=e$, by \eqref{eq:FC-equality} we have  $B_n(e)=(B_n)^J\subset \FC(B_n)$.

\item Now, let $\pi \in \FC(S_n)$ such thate $\pi\neq e$ and $1\notin\Des(\pi^{-1})$ (see Figure~\ref{examplesB-colored}, left). The diagonal reduced expression of $\pi$ takes the form  \eqref{diagonal-expression-A}.
%: $$\mathbf{\pi}=(s_{v}s_{v-1} \cdots s_{j_0})(s_{v_1}s_{v_1-1} \cdots s_{j_1})\cdots (s_{v_k}s_{v_k-1}\cdots s_{j_k}).$$
Now let $\mu$ be any element in $B_n(\pi)$. First notice that $\mu$ is FC, as $B_n(\pi)\subset (B_n)^J$. Moreover the diagonal reduced expression for $\mu$ is made of some factors chosen from the product $(s_0)(s_1s_0)\cdots (s_{v-1}s_{v-2}\cdots s_2s_1s_0)$. Assume that the rightmost of these factors is $(s_{i}s_{i-1}\cdots s_2s_1s_0)$, with $i\leq v-1$. Then we can concatenate the two expressions to get
$${\bf w}=(s_0)^\pm(s_1s_0)^\pm\cdots (s_{i}s_{i-1}\cdots s_2s_1s_0)\cdot (s_{v}s_{v-1}\cdots s_{j_0})(s_{v_1}s_{v_1-1} \cdots s_{j_1})\cdots (s_{v_k}s_{v_k-1}\cdots s_{j_k}),$$
where $(\cdot)^\pm$ means that the expression between the parentheses might appear or not appear. Recall that, as both $\mu$ and $\pi$ are separately FC, the two reduced expressions above for $\mu$ and $\pi$ do not contain any long braid type factor. We will show that their product is also a reduced expression for $w$ with no braid type factor. 

If a {\em nil factor} $s_i s_i$, or a long braid type factor appears in the product, it has to involve generators in the last factor of $\boldsymbol{\mu}$: $(s_{i}s_{i-1}\cdots s_1s_0)$ and in the first of $\boldsymbol{\pi}$: $(s_vs_{v-1}\cdots s_{j_0})$. Let us consider two consecutive occurrences of a generator $s_q$ in ${\bf w}$, one in $\boldsymbol{\mu}$ and the other in $\boldsymbol{\pi}$, therefore with $j_0\leq q\leq i$. By definition, the generator $s_{q-1}$ follows $s_q$ in the rightmost factor of $\boldsymbol{\mu}$. Moreover, since $q\leq i\leq v-1$ we have that $s_{q+1}$ appears before $s_q$ in $\boldsymbol{\pi}$. This implies that between the two occurrences of $s_q$ in $w$ there are both occurrences of $s_{q+1}$ and $s_{q-1}$. Hence neither nil nor braid type factor may appear in any expression of $w$, and $w\in \FC(B_n)$. 

\item Finally consider a FC element $\pi$ such that  $1\in\Des(\pi^{-1})$ (see Figure~\ref{examplesB-colored}, right). 
As observed before, either $\boldsymbol{\pi}=(s_1)(s_2)\cdots (s_j)$ or $\boldsymbol{\pi}=(s_1)(s_2)\cdots (s_{j_1})(s_{v}s_{v-1}\cdots s_{j_2})\mathbf{\tilde{u}}$. If $\mu=e$, then $\mu\cdot\pi=\pi$ is FC. If $\boldsymbol{\mu}=\delta_i$ for $1\leq i\leq v$, then the corresponding reduced expression for the product $\mu\cdot \pi$ is either equal to 
$$(s_{i-1}\cdots s_1s_0)\cdot (s_1)(s_2)\cdots (s_j)$$
or 
$$(s_{i-1}s_{i-2}\cdots s_1s_0)\cdot(s_1)(s_2)\cdots (s_{j_1})(s_{v}s_{v-1}\cdots s_{j_2})\mathbf{\tilde{u}}.$$
In both cases, this product is FC, since it contains neither nil nor braid factor. 
\end{enumerate}
\end{proof}

We call the set $B_n(\pi)\cdot \pi:=\{ \mu \cdot \pi \mid \mu \in B_n(\pi)\}$ the {\em fiber} associated to $\pi \in\FC(S_n)$.
We can characterize alternating and left-peak elements by using fibers.

\begin{corollary}\label{cor:LP_Alt} 
Let $w \in \FC(B_n)$ be written in the form $w=\mu \cdot \pi  \in (B_n)^J\times \FC(S_n)$ according to~\eqref{disjoint_union}. Then we have the following characterizations:
\begin{itemize}
    \item $w$ is a left-peak if and only if $\pi(1)=2$ and ${\mu}=\delta_i$ for some $i\in\{2,\dots,v(\pi)\}$;
    \item $w$ is alternating if and only if either $\pi(1)\neq2$ or ($\pi(1)=2$ and ${\mu}=e$ or ${\mu}=\delta_1$).
\end{itemize}
\end{corollary}

\begin{proof}
We start with the first assertion and assume that $\H(w)$ is a left-peak. Then the diagonal reduced expression (see Definition~\ref{def:diagonal} and~\eqref{diagonal-expression-LP}) for $w$ is of the form 
$$
\boldsymbol{w}=s_i\cdots s_1s_0s_1\boldsymbol{\pi}_0,
$$
where $i\geq1$ and $s_1\boldsymbol{\pi}_0$ is a reduced expression of an element in $\FC(A_{n-1})$. By uniqueness of the decomposition~\eqref{eq:decomposition}, we derive that reduced expressions for $\mu$ and $\pi$ are given by $\boldsymbol{\mu}=\delta_{i+1}$ and $\boldsymbol{\pi}=s_1\boldsymbol{\pi}_0$, respectively. Therefore $1\in\Des^L(\pi)=\Des(\pi^{-1})$, which by~\eqref{eq:pi(1)} is equivalent to $\pi(1)=2$. Setting $j=i+1\geq2$, we get the desired reduced expression for $\boldsymbol{\mu}$ by Theorem~\ref{thm:fibers}.

Conversely, again~\eqref{eq:pi(1)} implies that  $\boldsymbol{\pi}=s_1\boldsymbol{\pi}_0$. Therefore 
$$
\boldsymbol{w}=s_{i}\cdots s_1s_0s_1\boldsymbol{\pi}_0,
$$
where $1\leq i\leq v(\pi)-1$. As there is no $s_2$ between the above two occurrences of $s_1$, we deduce that $\H(w)$ is not alternating, so it is a left-peak.
\medskip

The second assertion is a consequence of the first one and Theorem~\ref{thm:fibers}, together with~\eqref{eq:pi(1)}.
\end{proof}

%\bj{\todo{Why is this example is useful ?
%\begin{example}
%Let $\pi=[2,4,5,1,3]=s_1s_3s_2s_4s_3$. Note that $\pi$ itself is not left peak, as well as $\mu \pi=[2,4,5,\bar{1},3]$, while for $\mu=s_1s_0$ we get $\mu \pi=[1,4,5,\bar{2},3]$ and for $\mu=s_2s_1s_0$  we have $\mu\pi=[1,4,5,\bar{3},2]$, both of them are left peaks. 
%\end{example}}}

To illustrate Theorem \ref{thm:fibers}, we end this section by giving two examples.

\begin{example}
Let $\pi=[1,5,2,3,4]=s_4s_3s_2$, therefore $\pi(1)\neq 2$ so $1 \not\in \Des(\pi^{-1})$. 
In the following table, for each element $\mu \cdot \pi$, both the one line notation and the diagonal reduced expression $\boldsymbol{\mu}$ are shown. All elements in the fiber are alternating. 
\begin{center}
\begin{tabular}{|c||r|r}
  \hline
    $B_5(\pi)\cdot \pi$  & $\boldsymbol{\mu}\cdot \boldsymbol{\pi}$\\ 
  \hline\hline
  
    $[1,5,2,3,4]$ & $(s_4s_3s_2)$  & \\
  \hline  
  $[\bar{1},5,2,3,4]$ & $s_0\cdot (s_4s_3s_2)$  &  \\
 \hline
  $[\bar{2},5,1,3,4]$ & $s_1s_0\cdot (s_4s_3s_2)$  \\
 \hline
$[\bar{3},5,1,2,4]$ & $s_2s_1s_0\cdot (s_4s_3s_2)$  &   \\
 \hline
  $[\bar{4},5,1,2,3]$ & $s_3s_2s_1s_0\cdot (s_4s_3s_2)$  &  \\
 \hline
  $[\bar{2},5,\bar{1},3,4]$ & $s_0s_1s_0\cdot (s_4s_3s_2)$  & \\
 \hline
  $[\bar{3},5,\bar{1},2,4]$ & $s_0s_2s_1s_0\cdot (s_4s_3s_2)$  & \\
 \hline
  $[\bar{4},5,\bar{1},2,3]$ & $s_0s_3s_2s_1s_0\cdot (s_4s_3s_2)$  &  \\
 \hline
  $[\bar{3},5,\bar{2},1,4]$ & $s_1s_0s_2s_1s_0\cdot (s_4s_3s_2)$  &  \\
 \hline
  $[\bar{4},5,\bar{2},1,3]$ & $s_1s_0s_3s_2s_1s_0\cdot (s_4s_3s_2)$  &\\
 \hline
  $[\bar{4},5,\bar{3},1,2]$ & $s_2s_1s_0s_3s_2s_1s_0\cdot (s_4s_3s_2)$  &    \\
 \hline
  $[\bar{3},5,\bar{2},\bar{1},2]$ & $s_0s_1s_0s_2s_1s_0\cdot (s_4s_3s_2)$  &   \\
 \hline
  $[\bar{4},5,\bar{2},\bar{1},3]$ & $s_0s_1s_0s_3s_2s_1s_0\cdot (s_4s_3s_2)$  &  \\
 \hline
  $[\bar{4},5,\bar{3},\bar{1},2]$ & $s_2s_1s_0s_3s_2s_1s_0\cdot (s_4s_3s_2)$  &    \\
 \hline
  $[\bar{4},5,\bar{3},1,2]$ & $s_2s_1s_0s_3s_2s_1s_0\cdot (s_4s_3s_2)$  &   \\
 \hline
  $[\bar{4},5,\bar{3},\bar{1},\bar{2}]$ & $s_2s_1s_0s_3s_2s_1s_0\cdot (s_4s_3s_2)$  &    \\
 \hline \hline

\end{tabular}
\end{center}
\medskip

Let $\pi=[2,4,5,1,3]=s_1s_3s_2s_4s_3$, therefore $\pi(1)= 2$ so $1 \in \Des(\pi^{-1})$. 
In this case the fiber is made of two alternating elements (the first two) and two left-peaks.

\begin{center}
\begin{tabular}{|r||r|r}
  \hline
    $B_5(\pi)\cdot \pi$  & $\boldsymbol{\mu}\cdot \boldsymbol{\pi}$ & \\
  \hline
  
    $[2,4,5,1,3]$ & $(s_1s_3s_2s_4s_3)$ &   \\
  \hline  
  $[2,4,5,\bar{1},3]$ & $s_0\cdot (s_1s_3s_2s_4s_3)$  &    \\
 \hline
  $[1,4,5,\bar{2},3]$ & $s_1s_0 \cdot (s_1s_3s_2s_4s_3)$  &  \\
 \hline

  $[1,4,5,\bar{3},2]$ & $s_2s_1s_0 \cdot (s_1s_3s_2s_4s_3)$   & \\
% \hline
 % $[135\bar{4}2]$ & $s_3s_2s_1s_0s_1s_3s_2s_4s_3$  & $[3,4][2,3][-1,3]$  & $\{4\}$   \\
 
  \hline \hline

\end{tabular}
\end{center}

\end{example}

\medskip

%%%%%%%%%%%%%%%%%%%%%%%%%%%%%%%%%%%%%%
\section{Cellular structure}\label{sec:Cellular}
%%%%%%%%%%%%%%%%%%%%%%%%%%%%%%%%%%%%%%

Recall the classical RSK bijection
from permutations in $S_n$ to pairs of standard Young tableaux of the same shape, see e.g.~\cite[\S 7.11]{EC2}.
This algorithm was extended to signed permutations in several ways, see, e.g.,~\cite{SW, Garfinkle}. 
%\todo{\bj{(see for example...). }}
Here we describe Barbash-Vogan's extension of the RSK algorithm which associates with each signed permutation $w \in B_n$ a pair of domino tableaux of the same shape
%due to Barbash and Vogan
~\cite{BV, Garfinkle}. We follow the exposition of \cite{Taskin}. 
%For a permutation $\pi\in S_n$ let $\shape(\pi)$ denote the common shape of the pair of 
%SYT which correspond to $\pi$.
We start with the following definition.
\begin{definition}
Let $w=[w_1,\dots,w_n]\in B_n$. The {\em palindromic representation} or the {\it $0$-core representation} of $w$ is $w^0: =[-w_n,\dots,-w_1,w_1,\dots,w_n] \in S_{[\pm n]}$.
\end{definition}
The first step of the algorithm applies the usual RSK algorithm on $w^0$ with respect to the natural order $-n<\cdots<-1<1<\cdots< n$
to get a pair of standard tableaux $P_0(w)$ and $Q_0(w)$. %Since the numbers in these tableaux are taken from the set $[\pm n]$, we use in 
In the second step we apply jeu de taquin slides to 
%straighten the tableaux with respect to the order $-1<1<-2<2<\cdots<-n<n$, that is, we
vacate each negative number $-i$ (starting from $-n$) 
in each of the Young tableaux $P_0(w)$ and $Q_0(w)$ until $-i$ becomes adjacent to $i$, in which case it loses its sign and becomes $i$. 
The resulting domino tableaux will be respectively denoted $\sP(w)$ and $\sQ(w)$. 

Here is an example which illustrates this algorithm.

\begin{example}
Let $w=[-3,1,2]$. Then $w^0=[-2,-1,3,-3,1,2]\in S_6$. By applying the RSK algorithm we get 

\[w^0\longmapsto (P_0(w), Q_0(w))=\left( \ \begin{ytableau}
 {-3} & {-1} & {1} & {2}\\
 {-2} & {3} \\
\end{ytableau} \quad
, \
\begin{ytableau}
 {-3} & {-2} & {-1} & {3}\\
 {1} & {2} \\
\end{ytableau} \ \right)
.\]
\smallskip

Now, the following process sticks the negative numbers to their positive counterparts by using jeu de taquin slides, starting by vacating $-3$, in both tableaux.

\[\begin{ytableau}
 {-3} & {-1} & {1} & {2}\\
 {-2} & {3} \\
\end{ytableau} \rightarrow
\begin{ytableau}
 {-2} & {-1} & {1} & {2}\\
 {3} & {3} \\
\end{ytableau} \rightarrow 
\begin{ytableau}
 {-1} & {-2} & {1} & {2}\\
 {3} & {3} \\  
\end{ytableau}\rightarrow 
\begin{ytableau}
 {-1} & {1} & {2} & {2}\\
 {3} & {3} \\
\end{ytableau}\rightarrow 
\begin{ytableau}
 {1} & {1} & {2} & {2}\\
 {3} & {3} \\
 \end{ytableau}=\sP(w)\]

and $$\begin{ytableau}
 {-3} & {-2} & {-1} & {3}\\
 {1} & {2} \\
\end{ytableau}\rightarrow
\begin{ytableau}
 {-2} & {-3} & {-1} & {3}\\
 {1} & {2} \\
\end{ytableau}\rightarrow
\begin{ytableau}
 {-2} & {-1} & {3} & {3}\\
 {1} & {2} \\
\end{ytableau}\rightarrow
\begin{ytableau}
 {1} & {-2} & {3} & {3}\\
 {1} & {2} \\
\end{ytableau}\rightarrow
\begin{ytableau}
 {1} & {2} & {3} & {3}\\
 {1} & {2} \\
\end{ytableau}=\sQ(w).$$
\end{example}

\medskip

\begin{proposition}\cite[Propositions 2.7 and 2.9]{Taskin}\label{prop: Taskin_Des}
%\bj{
The above extension of the RSK algorithm is a bijection between $B_n$ and pairs of domino tableaux such that for each $w \in B_n$, which 
satisfies  
%}
\begin{enumerate}
    \item $\sP(w^{-1})=\sQ(w)$ and $\sQ(w^{-1})=\sP(w)$,
    \item  $\Des_B(w)=\Des_B(\sQ(w))$ and $\Des_B(w^{-1})=\Des_B(\sP(w))$. 
\end{enumerate}
\end{proposition}

%Consider the natural embedding of $B_n$ in $S_{2n}$ as the set of permutations from
%$[-2n,2n]\setminus\{0\}$ to itself, which satisfy $\sigma(-i)=-\sigma(i)$.

%The following definition is the generalization of Knuth classes to signed permutations and domino tableaux, see~\cite{Lam, Taskin, Bonnafe}. 

\begin{definition}
The {\it two-sided %$\phi$-cell 
 combinatorial cell} of shape $\lambda\vdash 2n$
is the class
\[
{\mathcal C}_\lambda:=\{w \in B_n \mid \shape(\sP(w))=\lambda\}.
\]

\end{definition}

For an intensive discussion of these cells and their relations to the combinatorial description of the Kazhdan–Lusztig cells for type $B_n$ with unequal parameters, see~\cite{Lam, Taskin, Bonnafe}. 

\smallskip

%We call $\lambda$ the shape of the $\phi$-cell ${\mathcal C}_\lambda$.

%\bj{\todo{To verify citation in Thm 5.5. Is Thm 5.5 a Corollary of Prop 5.3 ? Why we need the first sentence of this Thm ? The sum is symmetric in the {\bf x} and {\bf y} or not ?}}

Recall that for $J\subseteq \{0,1,\ldots,n\}$ we denote ${\bf x}^J:=\prod_{i\in J}x_i$ and
${\bf y}^J:=\prod_{i\in J}y_i$.

%\todo{clarify citations below in theorem 5.5}

%The following theorem is essentially due to Barbash and Vogan~\cite{BV}. 

\begin{theorem}\label{thm:T}
For every partition $\lambda\in P^0(n)$ 
%$\lambda\vdash 2n$  with $f_2^\lambda\ne 0$
%there exists a $\Des_B$-preserving map 
%from pairs %elements in the $\phi$-cell 
%$\{(w,w^{-1}):\ w\in {\mathcal C}_\lambda\}$ to 
%pairs of domino tableaux in $\DSYT(\lambda)$. Hence
\begin{eqnarray}\label{eq:phi_cell-bitableaux}
\sum\limits_{w\in {\mathcal C_\lambda}} {\bf x}^{\Des_B(w)}{\bf y}^{\Des_B(w^{-1})} &=
\sum\limits_{(\sP,\sQ)\in \DSYT(\lambda)\times \DSYT(\lambda)}{\bf x}^{\Des_B(\sQ)}
{\bf y}^{\Des_B(\sP)}.
\end{eqnarray}
%SY domino tableaux of shape $\lambda$.
\end{theorem}

\begin{proof}
It follows from Proposition~\ref{prop: Taskin_Des}.
\end{proof}

\medskip

Green and  Losonczy proved that the  set $\FC(B_n)$ is a disjoint union of two-sided Kazhdan-Lusztig cells~\cite[Thm. 3.1.1]{GL}. 
%\bj{\todo{This sentence is obvious or not ? KL cells form a partitions of the whole groups $B_n$ so when you consider a restriction it remains a partition.}
%\todo{YR: Partitions in %$\phi$-cells 
%combinatorial cells is finer than partition in KL-cells}}
%In this section we prove the following 
We need a %more explicit version 
combinatorial analogue 
of this theorem. 

\medskip

%\smallskip

\begin{definition}\label{def:admissible}
The two following kinds of domino shapes will be called {\em admissible domino shapes}:
\begin{itemize}
\item $\lambda=(2n-k,k)$ for $0 \leq k \leq n$, 
\item $\lambda=((2n-2k,2k-1,1)$ for $1 \leq k \leq \lfloor n/2 \rfloor$. 
\end{itemize}
%Moreover, t
\end{definition}

\begin{theorem}\label{thm:cells}
%\begin{itemize}
%\item[1.] 
The set $\FC(B_n)$ is a disjoint union of two-sided %$\phi$-cells 
combinatorial cells 
of 
%bi-shapes $( (k),(n-k) )$, 
admissible domino shapes $(2n-k,k)$,
$0\le k\le n$, and 
%$( (n-k,k),\emptyset )$, $0<k\le \lfloor n/2\rfloor$.
$(2n-2k, 2k-1, 1)$, $1 \leq k\le \lfloor n/2\rfloor$. %In particular,
%\item[2.] For every $n\ge 1$
%\[
%FC(B_n)\setminus FC(S_n)= \bigsqcup_{k=1}^n {\mathcal C}_{(2n-k,k)}.
%\]
%\end{itemize}
\end{theorem}

\begin{proof}
For every element $w\in B_n$ the domino shape of $w$ is the 
common shape of the corresponding pair of domino tableaux. 
This is, in turn, the shape of the SYT corresponding to its palindromic 0-core representation  $w^0\in S_{[\pm n]}$ under the RSK bijection.
We have to show that $w\in \FC(B_n)$ if and only if its shape is admissible. Here, by the shape of $w$ we mean the shape of its image under the above extension of the RSK algorithm.

%\todo{Here is a purely combinatorial alternative proof.\\
%We will use basic properties of RSK and the bijection to domino tableaux and Stembridge's characterization of $\FC(B_n)$ elements via forbidden patterns~\cite[Theorem 5.1]{ST3}.
%\smallskip

\smallskip

First, we will show that the domino shape of any $w\in \FC(B_n)$ is admissible, namely
%that for every $w\in \FC(B_n)$ 
the height of the domino shape is $\le 3$ and the length of the second column %the corresponding its shape 
is $\le 2$.
Since the height of the shape of $\pi\in S_{2n}$ is equal to the maximal length of decreasing subsequence in $\pi$ (see~\cite{Sch61}), it suffices to 
prove Claim 1 below. 
%Next we will show that for every $w\in \FC(B_n)$ the length of the second column in %the corresponding 
%its shape is $\le 2$.   
By~\cite{Gre74}, the total length 
of the first two columns is the maximal size of two disjoint decreasing subsequence in $w^0$. Since no decreasing subsequence of length 4 appears in $w^0$,
it suffices to show that there are no two disjoint subsequences 
of length 3 in $w^0$. This will be proved in Claims 2 and 3.

\smallskip

\noindent{\bf Claim 1.} For every $w\in \FC(B_n)$,
there is no decreasing subsequence 
of length $4$ in $w^0$. 

\smallskip
\noindent{\it Proof of Claim 1.} 
Let  $(d,c,b,a)$ be a decreasing subsequence of length 4 in $w^0$. If the position of $b$ is in the right half then if $b<0$, the subsequence $(b,a)$ violates Proposition~\ref{prop:fc1}, since the pattern $[-1,-2]$ is forbidden; if $b>0$ then $0>-c>-d$ and $(-c,-d)$ is a forbidden subsequence in the right half. 
If the position of $b$ is in the left half then $(-b,-c,-d)$ is
a decreasing subsequence in the right half, violating Proposition~\ref{prop:fc1}, since the pattern $[3,2,1]$ is forbidden.
This proves that for every $w\in \FC(B_n)$ the height of the corresponding domino shape of $w$ is $\le 3$. Proof of Claim 1 is completed.

\smallskip

\noindent{\bf Claim 2.} For every $w\in \FC(B_n)$, any
decreasing subsequence of length 3 in $w^0$ is contained in a subsequence whose pattern is $[2,-1,1,-2]$ and having its last two letters in the right half of $w^0$.
%the last letter in the decreasing subsequence of lenth 3 must be the last letter the corresponding subsequence of length 4. 

\smallskip
\noindent{\it Proof of Claim 2.} 
%Without loss of generality, e
Notice that for every decreasing subsequence of length $3$ in $w^0$ $(c,b,a)$, $c>b>a$, the sequence $(-a,-b,-c)$ is also a decreasing subsequence of length $3$ in $w^0$.
 Without loss of generality, $(c,b,a)$ 
contains a decreasing subsequence of length $2$ in the right half,
since by definition of $w^0$, if the first two letters in  $(c,b,a)$ are in the left half then $(-a,-b,-c)$ is a decreasing subsequence in $w^0$ with two letters in the right half. Since $[3,2,1]$ is forbidden in $w$, we may assume that the %decreasing subsequence of length 3 
subsequence $(c,b,a)$  contains one letter in the left half of $w^0$.
Since $[-1,-2]$ is forbidden, the
second letter $b$ in the subsequence is positive, thus  $0<b<c$. 
%and hence first letter as well. 
If $c \ne -a$ then  the right half contains a shuffle of $-c$ with $(b,a)$, where $c>b>a$, violating Proposition~\ref{prop:fc1}.
%
%\todo{YR: Can we leave this to the reader?}
%
%If last letter $a$ satisfies $|a|<b$ 
%then the in the right half we have %$-c$; 
%%right of $b$, since $(-c,b,a)$ 
%a shuffle of $-c$ with $(b,a)$, where $c>b>a$, 
%violating Proposition~\ref{prop:fc1}.
%%, thus $(b,-c)$ is a subsequence of pattern $[1,-2]$ in the right half.
Thus $-a=c>b$ and both $(c,b,a)$ and $(-a,-b,-c)$ are  contained in $[-a,-b,b,a]$ with two letters in the right half.
%If the last letter $a$ satisfies $|a|>b$ then it must be negative
%and $(b,a)$ is a subsequence in the right half of pattern $[1,-2]$.
%%We conclude that an the existence of a decreasing subsequence
%%of length 3, forces the existence of a subsequence $(a,-b)$, with %$0<a<b$, in the right half. 
%%\todo{To be verified}
%Finally, notice that $c=-a$, otherwise the right half contains a shuffle of $-c$ with $(b,a)$, where $c>b>a$, violating Proposition~\ref{prop:fc1}. 
This completes the proof of Claim 2.

%\todo{To be verified and added: Notice that $c=-a$, otherwise the right half contains a shuffle of $-c$ with $(b,a)$, violating Proposition~\ref{prop:fc1}. This implies that all decreasing sub sequences of length 3 are contained in the pattern $[2,-1,1,-2]$ with last 2 letters in right half.
%This assures that $(a, -b)$ and $(x,-y)$ in the consequence are disjoint.} 

\smallskip
\noindent{\bf Claim 3.} For every $w\in \FC(B_n)$,
there are no two disjoint subsequences 
of length 3 in $w^0$. 

\smallskip
\noindent{\it Proof of Claim 3.} 
By Claim 2, the 
existence of two disjoint decreasing subsequences
of length 3 forces existence of two disjoint subsequences $(a,-b)$
and $(x,-y)$, with $0<a<b$ and $0<x<y$,  in the right half. Otherwise,
either $y=b$ or $x=a$. If $y=b$ both decreasing subsequences end with  $-b=-y$, thus not disjoint. If $x=a$ then, without loss of generality,  the right half contains
the subsequence $(a, -b, -y)$ of patterns $[1,-2,-3]$ or $[1,-3,-2]$; by  Proposition~\ref{prop:fc1}, both patterns are forbidden.
 
Consider all shuffles of $(a,-b)$ with $(x,-y)$ in the right half.
% we may assume, w.l.g., that 
Without loss of generality, the first letter in the shuffle is $a$, 
%positive and the last letter is negative. 
thus there exists a subsequence in the right half of $w^0$  of one of the following  
forms: 
$$(a,-b,x,-y)\;\mbox{or}\;(a,x,-b,-y)\;\mbox{or}\; (a,x,-y,-b).$$
Since $[3,2,1]$ is forbidden, we are forced to assume that $a<x$. 
Also, $[-1,-2]$ is forbidden, forcing $b>y$ in the two first above forms
and $y>b$ in the third one. %Putting all together 
Altogether we deduce that $a<x<y<b$
in the first two forms; and $a<b<y$, $a<x<y$ in the last one. Thus, the first two forms are of patterns
$[1,-4,2,-3]$ and $[1,2,-4,-3]$ respectively, and 
%in both cases last three letters violate one of the forbidden patterns in Proposition~\ref{prop:fc1}.
%\cite[Theorem 5.1]{Stembridge}, 
%both are forbidden. 
%For similar reasons, 
the pattern of the last form is either
$[1, 2, -4, -3]$ or $[1,3,-4,-2]$. In all cases, the last three letters violate one of the forbidden patterns in 
Proposition~\ref{prop:fc1}. This completes the proof of Claim 3.

\smallskip

\smallskip

We have proved that every $w\in \FC(B_n)$ has an admissible shape.
To finish the proof, notice that
by Corollary~\ref{cor:degree}, the sum of the squares of the number of domino tableaux of all admissible shapes is equal to
\begin{equation}\label{eq:total}
\sum\limits_{k=0}^n (f_2^{(2n-k,k)})^2+\sum\limits_{k=1}^{\lfloor n/2\rfloor}(f_2^{(2n-2k,2k-1,1)})^2= \frac{n+2}{n+1}\binom{2n}{n}-1,
\end{equation}
which is equal by Remark~\ref{number of alternatings and LP} to  
the number of elements in $\FC(B_n)$. One concludes that
%By Equation~\eqref{eq:total}, 
the total size of combinatorial two-sided cells of 
%number of elements in $B_n$ with 
admissible shapes in $B_n$ 
is equal to the size of $\FC(B_n)$. 
%Hence, it suffices to show that every $w\in \FC(B_n)$ has an admissible shape.
Since all FC elements are of admissible shapes,
this shows that there are no elements in $B_n\setminus \FC(B_n)$ whose shape is admissible, completing the proof.

\end{proof}

\begin{remark}\label{rem:BGIL} It was conjectured by Bonnaf\'e, Geck, Inacu, and Lam~\cite[Conjecture A(c)]{Lam} that   Kazhdan-Lusztig two-sided cells of type $B_n$ with unequal parameters are  two-sided combinatorial cells, see also~\cite{Bonnafe}. By Green-Losonczy Theorem~\cite[Thm. 3.1.1]{GL}, Theorem~\ref{thm:cells} (with no explicit description of the shapes) follows from this conjecture.
\end{remark}

We conclude now that the Barbash--Vogan bijection described above preserves the division of the set of FC elements into alternating and left-peak elements. 

\begin{corollary}\label{cor:shapesLP}
We have the following equivalences:
\begin{itemize}
    \item  $w\in \FC(B_n)$ is a left-peak if and only if $shape(\sP(w))=(2n-2k,2k-1,1)$, for some  $k$.
    \item $w\in \FC(B_n)$ is alternating if and only if $shape(\sP(w))=(2n-k,k)$, for some $k$.
\end{itemize}
\end{corollary}

\begin{proof}
Let $w$ be a left-peak. By Remark~\ref{one line notation of left peak}, the one line notation of $w$ is $w=[1,\dots,-k,\dots]$ for an integer $k>1$. Thus $w^0=[\dots,k,\dots,-1,1,\dots, -k,\dots]$, containing a $[3,2,1]$-pattern, hence the height of the shape of $\sP(w)$ is at least $3$. By Theorem~\ref{thm:cells} and the cardinality arguments given in Remark~\ref{number of alternatings and LP}, we get the first equivalence.
 The second assertion is then a consequence of the first one and Theorem~\ref{thm:cells}.
 \end{proof}

%\bj{\todo{
\begin{remark} Three different decompositions of FC$(B_n)$ into disjoint subsets are considered in the present paper. 
The first one into Kazhdan-Lusztig cells is due to Green-Losonczy~\cite{GL}; the second one into Barbash-Vogan combinatorial cells is given in Theorem~\ref{thm:cells}; the third one into fibers is shown in Theorem~\ref{thm:fibers}. Comparing Corollary~\ref{cor:shapesLP} with Corollary~\ref{cor:LP_Alt}, one deduces  
that fibers are, in general, different from combinatorial cells. 
A remaining open problem is whether Kazhdan-Lusztig cells and combinatorial cells, restricted to $\FC(B_n)$, coincide or not. % remains an open problem. As Yuval wrote 
A positive answer to this question would solve a special case of~\cite[Conjecture A(c)]{Lam} mentioned in Remark~\ref{rem:BGIL}.
\end{remark}

%Do we want to write some of this remark here or not ?}}

%%%%%%%%%%%%%%%%%%%%%%%%%%%%%%%%%%%%%%%%%%%%%%%%
\section{Equidistribution}\label{sec:equidistribution}
%%%%%%%%%%%%%%%%%%%%%%%%%%%%%%%%%%%%%%%%%%%%%%%%

%The goal of this section is to prove an equidistribution result over $\FC(B_n)$, which will be crucial to derive our main theorem. 
In this section we prove Theorem~\ref{thm:2}.
In order to do this, we introduce an involution on $\FC(B_n)$, relying on the decomposition into fibers from Section~\ref{sec:fibers} and on the properties of an involution due to Rubey. 
Throughout this section, for a signed permutation $w$, we set $\Des(w)$ as in~\eqref{eq:DesA}. 
\medskip

In~\cite{Rubey}, Rubey defines an involution $f: S_n(321)\longrightarrow S_n(321)$ satisfying the following properties.
\begin{proposition}\label{prop:Rubey}
For each $\pi\in \FC(S_n)$, we have
\begin{itemize}
\item[$i)$] $\Des(\pi)=\Des(f(\pi))$; 
\item[$ii)$] $\bl(f(\pi)^{-1})=n-\ldes(\pi^{-1})$, equivalently \  $\bl(\pi^{-1})=n-\ldes(f(\pi)^{-1})$. 
%\br{\item[$iii)$]{Des}$(\pi)=$ {\rm Des}$(f(\pi));$} 
%\item[$iv)$] $\bl(\pi)=\bl(\pi^{-1})$.
\end{itemize}

%where for $\pi \in S_n$:
%\[
%LTRM(\pi):=\{i:\ \pi(i) = \max\{\pi(1), \ldots, \pi(i)\}\}
%\]

\end{proposition}

We extend the involution $f$ to a mapping $ \Phi :\FC(B_n)\longrightarrow B_n$ by defining 
\begin{equation}\label{def:Phi}
\Phi(w):=f(\pi^{-1}) \cdot \mu^{-1},
\end{equation}
where $w= \pi^{-1}\cdot\mu^{-1}$, or equivalently $w^{-1}=\mu\cdot \pi$, is the decomposition in~\eqref{eq:decomposition}. Note that, as recalled in Section~\ref{sec:reduced-expressions}, $\pi^{-1}\in S_n(321)$, so the map $\Phi$ is well defined. 

\begin{observation}\label{cor:ldes}
%For $w\in B_n$ let $w^{-1}=\mu\cdot \pi$ be the decomposition in Equation~\eqref{eq:decomposition}. 
Let $J=S\setminus \{s_0\}$ and recall the definition of the quotient $(B_n)^J$ from~\eqref{quot}.
For every $\mu\in (B_n)^J$, $\Des(\mu)=\emptyset$, so  
left multiplication by $\mu$ %an increasing sequence 
is order preserving. Hence, for every $\pi$ in the parabolic subgroup $(B_n)_J\cong S_n$  
\[
\Des(\mu\cdot \pi)=\Des(\pi),
\]
thus $\ldes(\mu\cdot \pi)=\ldes(\pi)$. 
\end{observation}

%\begin{lemma}\label{cor:ldes}
%Set $\pi \in \FC(S_n)$ with $\pi \neq e$, and $\mu \in B_n(\pi)$. Then
%\begin{equation*}\label{ldespres}
%\ldes(f(\pi)^{-1})=\ldes(\mu \cdot f(\pi)^{-1})
%\end{equation*}
%\end{lemma}

%\begin{proof}
%%We have $\Des(\pi)\neq \emptyset$, which implies that for 
%For any $\mu \in B_n(\pi)$  %each descent of $\pi$ is also a descent of $\mu\cdot\pi$, 
%\[
%\Des(\pi)=\Des(\mu\cdot \pi),
%\]
%as left multiplication by an increasing sequence is order preserving. 
%Since $\pi\neq e$, by Proposition~\ref{prop:Rubey} it follows that $\Des((f(\pi^{-1})^{-1})\neq \emptyset$ and we get the result.
%\end{proof}

\begin{lemma}\label{v-preserving} 
For every $\pi \in \FC(S_n)$,
%$$\Des(\pi^{-1})=\Des(f(\pi^{-1}))$$.
%In particular, 
$B_n(\pi)=B_n(f(\pi^{-1})^{-1})$.
%$v(\pi)=v(f(\pi^{-1})^{-1})$.
\end{lemma}

\begin{proof}
By definition, $\Des^L(\pi)=\Des(\pi^{-1})$, thus
by Proposition~\ref{prop:Rubey}, $\Des^L(\pi)=\Des^L(f(\pi^{-1})^{-1})$. 
Since, by Theorem~\ref{thm:fibers}, $B_n(\pi)$ depends only on the left descent set of $\pi$, 
the statement holds.
%and $v(\pi)=v(f(\pi^{-1})^{-1})$ and the lemma follows. 
\end{proof}

\begin{lemma}\label{Des-preserving} 
For every $\pi\in \FC(S_n)$ and $\mu\in (B_n)^J$,
\[
\Des_B(\pi^{-1}\cdot \mu^{-1})=\Des_B(f(\pi^{-1})\cdot \mu^{-1}), 
\]
equivalently,
%Let $\pi \in \FC(S_n)$ and $\mu \in B_n(\pi)$. Then 
$\Des_B^L(\mu \cdot \pi)=\Des_B^L(\mu \cdot f(\pi^{-1})^{-1})$.
\end{lemma}

\begin{proof}

Let $\mu \in (B_n)^J$. Observe that for each $\pi \in S_n$, if $\mu^{-1}(i)<0$ and $\mu^{-1}(i+1)>0$ ($\mu^{-1}(i)>0$ and $\mu^{-1}(i+1)<0$) then $i\not\in \Des(\pi\cdot \mu^{-1})$ ($i\not\in \Des(\pi\cdot \mu^{-1})$) independently of $\pi \in S_n$. 
On the other hand, if $\mu^{-1}(i)$ and $\mu^{-1}(i+1)$ have the same sign then by Observation~\ref{obs:mu}, $\mu^{-1}(i),\mu^{-1}(i+1)$ have consecutive values. Then for each $\pi \in S_n$,
$i\in \Des(\pi\cdot \mu^{-1})$ if and only if
$0<\mu^{-1}(i) \in \Des(\pi)$ or $0<-\mu^{-1}(i+1)\in \Des(\pi)$.
%If $j<0$ then  $i\in \Des(\pi\cdot \mu^{-1}$ if and only if $-j\in \Des(\pi)$.
It follows that in the last case, $i \in \Des(\pi^{-1}\cdot \mu^{-1})$ if and only $\mu^{-1}(i) \in \Des(\pi^{-1})$. By Proposition ~\ref{prop:Rubey}, this is true if and only if  $\mu^{-1}(i) \in \Des(f(\pi^{-1}))$. 
We deduce:
\[
\Des(\pi^{-1}\cdot \mu^{-1})=\Des(f(\pi^{-1})\cdot \mu^{-1}).
\]

To conclude, notice that
\[
0\in \Des_B(\pi^{-1}\cdot \mu^{-1}) \Longleftrightarrow \mu^{-1}(1)<0 \Longleftrightarrow  0\in \Des_B(f(\pi^{-1})\cdot \mu^{-1}). 
\]

\end{proof}

\begin{proposition} \label{thm:Phi} \
%Let us consider 
For every $w \in \FC(B_n)$ %. Then :
\begin{itemize}
%\item[(i)] $\Phi(B_n(\pi)\cdot\pi)=B_n(\pi)\cdot f(\pi^{-1})^{-1}$.
\item[(i)] $\Phi(w)\in \FC(B_n)$;
\item[(ii)]  $\Des_B(w)=\Des_B(\Phi(w))$;
\item[(iii)] $\bl(w^{-1})=n-\ldes(\Phi(w)^{-1})$;
\item[(iv)] ${\rm Neg}(w)={\rm Neg}(\Phi(w))$.
\end{itemize}
\end{proposition}

\begin{proof} \
\begin{itemize}
\item[(i)]  Following~\eqref{disjoint_union}, we write uniquely $w^{-1}=\mu\cdot\pi$
with $\mu\in B_n(\pi)$.  
%We have 
By Lemma~\ref{v-preserving},
\[
\mu\in B_n(\pi)= B_n(f(\pi^{-1})^{-1}).
\]
%where
%%$\mu$ belongs to the set $B_n(\pi)$, By 
%the equality follows from Lemma~\ref{v-preserving}.
%,  $B_n(\pi)$ is equal to the set $B_n(f(\pi^{-1})^{-1})$. 
By definition, 
\[
\Phi(w)^{-1}=\mu \cdot f(\pi^{-1})^{-1}\in %$ which belongs to the set $
B_n(\pi)\cdot f(\pi^{-1})^{-1} = B_n(f(\pi^{-1})^{-1})  \cdot f(\pi^{-1})^{-1} \subseteq \FC(B_n).
\]
The last containment follows from Theorem~\ref{thm:fibers}. 
Hence, as mentionned in Section~\ref{sec:reduced-expressions},
$\Phi(w)\in \FC(B_n)$.
\smallskip

\item[(ii)] One can write 
 $$ \Des_B(w) = \Des_B^L(w^{-1})= \Des_B^L(\mu\cdot\pi) = \Des_B^L(\mu \cdot f(\pi^{-1})^{-1})=\Des_B((f(\pi^{-1})\cdot \mu^{-1})=\Des_B(\Phi(w)),$$  
 where the third equality follows from Lemma~\ref{Des-preserving} .
\smallskip

\item[(iii)] The following equalities are derived from Proposition~\ref{prop:Rubey} and Observation~\ref{cor:ldes}:
\begin{eqnarray*}
\bl(w^{-1})=\bl(\mu\cdot \pi)=\bl(\pi)&=&n-\ldes(f(\pi^{-1})^{-1})\\
&=&n-\ldes(\mu \cdot f(\pi^{-1})^{-1})\\
&=&n-\ldes(\Phi(w)^{-1}).
\end{eqnarray*}

\smallskip

\item[(iv)] Multiplying a signed permutation on the left by a permutation in $S_n$ does not change the positions of the negative entries.  Hence the result follows from the definition of $\Phi$.
\end{itemize}
\end{proof}

\begin{remark}\label{involution}
As $f$ is an involution, by Theorem~\ref{thm:Phi}(i),  the map $\Phi$ is an involution on $\FC(B_n)$. 
\end{remark}

Now we are ready to prove our equidistribution result given in Theorem~\ref{thm:2}.

\begin{proof}[Proof of Theorem~\ref{thm:2}]
By Remark~\ref{involution} together with Proposition~\ref{thm:Phi}(ii)--(iv),
$\Phi$ is an involution on $\FC(B_n)$ which maps the left hand side 
to the right hand side.
%As $\Phi$ is a bijection on $\FC(B_n)$, one can replace $w$ by $\Phi(w)$ in all exponents on the left hand side. By Proposition~\ref{thm:Phi}(ii)--(iv) and the fact that $\Phi$ is an involution we get the result. 
\end{proof}

%%%%%%%%%%%%%%%%%%%%%%%%%%%%%%%%%%%%%%%%%%%%%%%%%
\section{Proof of the main theorem}\label{sec:proofMain}

Applying the vector space homomorphism from the ring of quasi-symmetric functions to the multilinear subspace of the formal power series ring $\ZZ[x_1,x_2,\ldots]$, defined by $F_{n,J}\mapsto {\bf x}^J$, and using the fact that for every $\pi\in S_n$, $\bl(\pi^{-1})=\bl(\pi)$,  
%it suffices to prove the following statement, which implies %is equivalent to 
Theorem~\ref{thm:ABR} is equivalent to the following equation
%which describes the equidistribution of the FC elements in $S_n$ and the pairs of tableaux of certain shapes. 
\begin{equation}\label{eq:A}
\sum_{\pi \in \FC(S_n)}{\bf x}^{\Des(\pi)}q^{\bl(\pi^{-1})}=
\sum_{k=0}^{\lfloor n/2 \rfloor} \sum_{(P,Q)\in \SYT^2(n-k,k)}{\bf x}^{\Des(Q)} q^{n-\ldes(P)},
\end{equation}
where the right sum is over pairs of tableaux of shape $(n-k,k)$. 

In this section we use Theorem~\ref{thm:2} to prove the type $B$ analogue, namely 
Theorem~\ref{thm:main}. 
The first step is to present the corresponding equidistribution for type $B$ in the language of domino tableaux.

\begin{theorem}\label{thm:3} For any positive integer $n$ we have 
%[Explicit version of Thm.~\ref{thm:1}]
\begin{multline}
\sum\limits_{w\in \FC(B_n)} {\bf x}^{\Des_B(w)}q^{\bl(w^{-1})} =
\sum\limits_{k=0}^n \sum\limits_{(\sP,\sQ)\in \DSYT^2(2n-k,k)}{\bf x}^{\Des_B(\sQ)}q^{n-\ldes(\sP)}.
\\
+ \sum\limits_{k=1}^{\lfloor n/2\rfloor} \sum\limits_{(\sP,\sQ)\in \DSYT^2(2n-2k,2k-1,1)}{\bf x}^{\Des_B(\sQ)}
q^{n-\ldes(\sP)}.
\end{multline}
\end{theorem}

\begin{proof}%[Proof of Corollary~\ref{thm:3}]

Recall from Theorem~\ref{thm:cells} that the set $\FC(B_n)$ is a union of %$\phi$-cells 
combinatorial cells 
corresponding to 
the domino tableaux of the shapes $(2n-k,k)$ and $(2n-2k,2k-1,1)$. 

By Theorem~\ref{thm:cells} together with Equation~\eqref{eq:phi_cell-bitableaux},
%Theorem~\ref{prop:phi_cell-bitableaux},
\[
\begin{aligned}
\sum\limits_{w\in \FC(B_n)} {\bf x}^{\Des_B(w)}{\bf y}^{\Des_B(w^{-1})} &=
\sum\limits_{k=0}^n \sum\limits_{(\sP,\sQ)\in \DSYT^2(2n-k,k)}{\bf x}^{\Des_B(\sQ)}
{\bf y}^{\Des_B(\sP)}
\\
&+ \sum\limits_{k=1}^{\lfloor n/2\rfloor} \sum\limits_{(\sP,\sQ)\in \DSYT^2(2n-2k,2k-1,1)
}{\bf x}^{\Des_B(\sQ)}
{\bf y}^{\Des_B(\sP)}.
\end{aligned}
\]
Applying the map  ${\bf y}^J\mapsto q^{n-j_t}$, where $J=\{j_1<j_2<\cdots<j_t\}\subseteq [0,n-1]$, together with Theorem~\ref{thm:2} completes the proof. 
\end{proof}

Next, we deduce the following consequence, which is the translation of Theorem~\ref{thm:3} to the language of bi-tableaux. In order to give a more elegant version of this result, we consider here the equidistribution over $\FC(B_n) \setminus \FC(S_n)$ rather than over $\FC(B_n)$. 

\begin{corollary}\label{thm:4} For any positive integer $n$ we have
\begin{multline}\label{eq:4}
\sum_{w\in \FC(B_n)\setminus \FC(S_n)} {\bf x}^{\Des_B(w)}q^{\bl(w^{-1})} =
\sum_{k=1}^{\lfloor n/2\rfloor}\sum_{(P,Q)\in \BSYT^2((k),(n-k))}
{\bf x}^{\Des_B(Q)}q^{n-\ldes(P)}\\
+\sum_{k=0}^{\lfloor (n-1)/2\rfloor}\sum_{(P,Q)\in \BSYT^2((n-k),(k))}{\bf x}^{\Des_B(Q)}q^{n-\ldes(P)}.
\end{multline}
\end{corollary}

\begin{proof}
%By Corollary~\ref{thm:V}, we have that 
%It follows that
By Remark~\ref{rem:bijections},
\begin{eqnarray*}
\sum\limits_{k=0}^{n} \sum\limits_{(\sP,\sQ)\in \DSYT^2(2n-k,k)}
{\bf x}^{\Des_B(\sQ)}q^{n-\ldes(\sP)} 
& = &\sum\limits_{k=0}^{\lfloor n/2\rfloor}\sum\limits_{(\sP,\sQ)\in \DSYT^2(2n-2k,2k)}
{\bf x}^{\Des_B(\sQ)}q^{n-\ldes(\sP)}\\
& + & \sum\limits_{k=0}^{\lfloor (n-1)/2\rfloor}{\sum\limits_{(\sP,\sQ)\in \DSYT^2(2n-2k-1,2k+1)}
{\bf x}^{\Des_B(\sQ)}q^{n-\ldes(\sP)}}\\
& = &\sum\limits_{k=0}^{\lfloor n/2\rfloor}{\sum\limits_{(P,Q)\in \BSYT^2((k),(n-k))}
{\bf x}^{\Des_B(Q)}q^{n-\ldes(P)}}\\
& + &\sum\limits_{k=0}^{\lfloor (n-1)/2 \rfloor}{\sum\limits_{(P,Q)\in \BSYT^2((n-k),(k))}
{\bf x}^{\Des_B(Q)}q^{n-\ldes(P)}}
\end{eqnarray*}

and 

\begin{eqnarray*}
\sum\limits_{k=1}^{\lfloor n/2\rfloor}
\sum\limits_{(P,Q)\in \DSYT^2(2n-2k,2k-1,1)}
{\bf x}^{\Des_B(Q)}q^{n-\ldes(P)}
& = &\sum\limits_{k=1}^{\lfloor n/2 \rfloor} 
\sum\limits_{(P,Q)\in \BSYT^2(\emptyset,(n-k,k))}
{\bf x}^{\Des_B(Q)}q^{n-\ldes(P)}\\
& = &\sum\limits_{k=1}^{\lfloor n/2\rfloor} 
\sum\limits_{(P,Q)\in \SYT^2(n-k,k)}{\bf x}^{\Des_B(Q)}q^{n-\ldes(P)},
\end{eqnarray*}
where the last equality is due to the obvious descent-preserving bijection between 
$\BSYT(\emptyset,(n-k,k))$ and $\SYT(n-k,k)$. 
Now, by Theorem~\ref{thm:3} and Equation~\eqref{eq:A}, we obtain

\begin{eqnarray*}
\sum_{w\in \FC(B_n)\setminus \FC(S_n)} {\bf x}^{\Des_B(w)}q^{\bl(w^{-1})} & = &
\sum_{w\in \FC(B_n)} {\bf x}^{\Des_B(w)}q^{\bl(w^{-1})}-
\sum_{w\in \FC(S_n)} {\bf x}^{\Des_B(w)}q^{\bl(w^{-1})} \\
& = &
\sum_{k=0}^{\lfloor n/2\rfloor}\sum_{(P,Q)\in \BSYT^2((k),(n-k))}{\bf x}^{\Des_B(Q)}q^{n-\ldes(P)} \\
& + &
\sum\limits_{k=0}^{\lfloor (n-1)/2 \rfloor}{\sum\limits_{(P,Q)\in \BSYT^2((n-k),(k))}
{\bf x}^{\Des_B(Q)}q^{n-\ldes(P)}}\\
& + &
\sum_{k=1}^{\lfloor n/2\rfloor}\sum_{(P,Q)\in \SYT^2(n-k,k)}
{\bf x}^{\Des_B(Q)}q^{n-\ldes(P)}
\\
& - &\sum\limits_{k=0}^{\lfloor n/2\rfloor} 
\sum\limits_{(P,Q)\in \SYT^2(n-k,k)}{\bf x}^{\Des_B(Q)}
q^{n-\ldes(P)},
\end{eqnarray*}
which is equal to the RHS of Equation~\eqref{eq:4}.
\end{proof}

\begin{proof}[Proof of Theorem~\ref{thm:main}]
By Corollary~\ref{thm:4} we have 
\begin{multline*}
\sum\limits_{w\in \FC(B_n)\setminus \FC(S_n)} {\bf x}^{\Des_B(w)}q^{\bl(w^{-1})}  = 
\sum_{k=1}^{\lfloor n/2\rfloor} \left (\sum_{P \in \BSYT((k),(n-k))} q^{n-\ldes(P)} \right) \left( \sum_{Q \in \BSYT((k),(n-k))} {\bf x}^{\Des_B(Q)}\right )\\
+ \sum_{k=0}^{\lfloor (n-1)/2\rfloor}\left (\sum_{P \in \BSYT((n-k),(k))} q^{n-\ldes(P)} \right) \left( \sum_{Q \in \BSYT((n-k),(k))} {\bf x}^{\Des_B(Q)} \right).
\end{multline*}
 %Corollary~\ref{thm:V}.1 
 Remark~\ref{rem:bijections} then transforms both sums over $Q$ on the right-hand side of the above identity to sums over $\sQ\in \DSYT(2n-2k,2k)$ and $\sQ\in \DSYT(2n-2k-1,2k+1)$, respectively.
Applying the vector space homomorphism from the multi-linear subspace of the formal power series ring $\ZZ[x_0,x_1,x_2,\ldots]$
to the ring of Chow's type $B$ quasi-symmetric functions, defined by ${\bf x}^J\mapsto F^B_{n,J}$, to both sides of the resulting formula, we can then use
Proposition~\ref{thm:schurB} to transform the sums over $\sQ$ and get:
\begin{multline*}
\sum\limits_{w \in \FC(B_n)\setminus \FC(S_n)} q^{\bl(w^{-1})} F^B_{\Des_B(w)}= 
\sum_{k=1}^{\lfloor n/2\rfloor} \left (\sum_{P \in \BSYT((k),(n-k))} q^{n-\ldes(P)} \right) s_{(k)}(x_1,x_2,\ldots)\ s_{(n-k)}(x_0,x_1,\ldots)\\
+ \sum_{k=0}^{\lfloor (n-1)/2\rfloor}\left (\sum_{P \in \BSYT((n-k),(k))} q^{n-\ldes(P)} \right) s_{(n-k)}(x_1,x_2,\ldots)\ s_{(k)}(x_0, x_1,\ldots).
\end{multline*}
The conclusion follows by replacing $k$ by $n-k$ in the second sum and noting that every $P \in \BSYT((k),(n-k))$ 
may be identified with a $T\in \SYT((n,k)/(k))$ 
with same $\ldes$.

%\todo{YR: Are statements equivalent ? Is Chow's QSF a basis ?}

\end{proof}

%\todo{Give precisions for ref [3] below }

%\section{Final remarks and open problems}

%\begin{remark}
%Note that for domino shapes of the form  $\lambda=(2n-2k,2k-1,1)$, the Carr\'e-Leclerc bijection is $\Des_B$-preserving. Indeed there is a unique way to fill the partition $\lambda=(2n-2k,2k-1,1)$ with dominoes: there must be a unique vertical domino at the bottom of the first column, and all other dominoes have to be horizontal. The image of $\lambda$ by the algorithm is $(\emptyset, T^+)$, where the entries in the boxes of $\sT^+$ are those of the corresponding dominoes of $\sT$ in the same relative order. Hence descents are preserved. For example,
%\[
%\ytableaushort{112255,344,3} \ \longmapsto \ytableaushort{125,34} .
%\]
%\end{remark}

%%%%%%%%%%%%%%%%%%%%%%%%%%%%%%%%%%%%%%%%%%%%%%%%%%%%%%%%%%%%%%%
\section{Two notions of type B Schur-positivity}\label{sec:types}
%%%%%%%%%%%%%%%%%%%%%%%%%%%%%%%%%%%%%%%%%%%%%%%%%%%%%%%%%%%%%%%

A subset $A\subseteq S_n$ is {\em Schur-positive} if the quasisymmetric function $\Q(A):=\sum_{w\in A}F_{\Des(w)}$ 
is symmetric and Schur-positive.
Here $\{F_J\mid J\subseteq [n-1]\}$ are Gessel's fundamental quasi-symmetric functions
%indexed by  $J\subseteq [n-1]$ 
%. Recall  
and $\Des(w)$ is the standard descent set from~\eqref{eq:DesA}.
Determining
whether a given symmetric function is Schur-positive is a major problem in
contemporary algebraic combinatorics~\cite{Stanley_problems}.

\smallskip 

As mentioned in the introduction, the concept of quasi-symmetric functions has been extended to Coxeter groups of type $B$ in two different
ways. The two associated notions of type $B$ Schur-positivity follow.

\medskip

Recall Chow's type $B$ fundamental quasi-symmetric functions $\{F^B_J \mid J\subseteq \{0\}\cup [n-1]\}$ from Definition~\ref{def:Chow_fundamental} and the type $B$ right descent set $\Des_B(\pi)$ from~\eqref{eq:DesB}. Definition~\ref{def:domino_function} of domino functions 
%Using Chow's approach, 
leads to the following type $B$ Schur-positivity notion, introduced in~\cite{MV4}.

\begin{definition}
A subset $A\subseteq B_n$ is {\em Chow type $B$ Schur-positive} if
the Chow  quasi-symmetric function
\[
\Q^C(A):=\sum\limits_{w\in A} F^B_{\Des_B(w)}
\]
can be written as a non-negative sum of 
domino functions.
\end{definition}

\begin{proposition}\label{cor:Chow_FC}
For every $n\ge j\ge 1$, the set $\{w\in \FC(B_n)|\ \bl(w^{-1})=j\}$ is Chow type $B$ Schur-positive.
\end{proposition}

We will first prove the following lemma.
Consider the natural embedding of $S_n$ in $B_n$. 

\begin{lemma}\label{lem_GC}
Let $A\subseteq S_n\subseteq B_n$ and $\Q(A)=\sum_{w\in A} F_{\Des(w)}(x_1,x_2,\dots)$. If $\Q(A)$ is symmetric in $x_1,x_2,\dots$
then $\Q^C(A)$ is symmetric in $x_0,x_1,\dots$, and for every $\lambda\vdash n$
\[
\langle \Q^C(A), s_\lambda(x_0,x_1,\dots)\rangle= \langle \Q(A), s_\lambda(x_1,x_2,\dots)\rangle,
\]
where $\langle \cdot, \cdot\rangle$ is the standard scalar product
on symmetric functions. 
%\[
%\Q(A):=\sum\limits_{w\in A} F_{\Des(w)}(x_1,x_2,\dots)=\sum\limits_{\lambda\vdash n} c_\lambda\  s_\lambda(x_1,x_2,\dots)
%\]
%then
%\[
%\Q^C(A)=\sum\limits_{w\in A} F^B_{\Des_B(w)}(x_0,x_1,\dots)=\sum\limits_{\lambda\vdash n} c_\lambda\  s_\lambda(x_0,x_1,\dots).
%\]
\end{lemma}

\begin{proof}
%Consider the natural embedding of $S_n$ in $B_n$, and n
For all $\lambda\vdash n$, consider
\[
c_\lambda:=\langle \Q(A), s_\lambda(x_1,x_2,\dots)\rangle.
\]
By assumption, $\Q(A)$ is symmetric in $x_1,x_2,\dots$. Thus, by  Theorem~\ref{thm:schur},
\[
\sum\limits_{J\subseteq [n-1]} a_{A,J} F_J=
\Q(A)=
%\sum\limits_{w\in A} F_{\Des(w)}(x_1,x_2,\dots)=
\sum\limits_{\lambda\vdash n} c_\lambda\ s_\lambda(x_1,x_2,\dots)=
\sum\limits_{\lambda\vdash n} c_\lambda \sum\limits_{T\in \SYT(\lambda)}F_{\Des(T)}=\sum\limits_{\lambda\vdash n} c_\lambda \sum\limits_{J\subseteq [n-1]} b_{\lambda,J}F_J,
\]
where $a_{A,J}:=\#\{w\in A\mid \Des(w)=J\}$ and
$b_{\lambda,J}:=\#\{T\in \SYT(\lambda)\mid \Des(T)=J\}$.
It follows that
\[
a_{A,J}= \sum\limits_{\lambda\vdash n} c_\lambda b_{\lambda,J}.
\]
Now notice that
for every $w\in S_n$, $0 \not\in \Des_B(\pi)$, thus 
%the type $B$ right descent set $\Des_B(\pi)$ from~\eqref{eq:DesB} is equal to the standard type $A$ descent set 
$\Des_B(w)=\Des(w)$ and $\#\{w\in A\mid \Des_B(w)=J\}=a_{A,J}$. 
It follows that
\begin{eqnarray*}
\Q^C(A)&=&\sum\limits_{w\in A} F^B_{\Des_B(w)}=\sum\limits_{J\subseteq [n-1]} a_{A,J} F^B_{J}= \sum\limits_{J\subseteq [n-1]} \sum\limits_{\lambda\vdash n} c_\lambda b_{\lambda,J} F^B_J\\
&=& \sum\limits_{\lambda\vdash n} c_\lambda
 \sum\limits_{J\subseteq [n-1]} b_{\lambda,J} F^B_J= \sum\limits_{\lambda\vdash n} c_\lambda\ s_{\lambda}(x_0,x_1,\dots).
\end{eqnarray*}
The last equality follows from Proposition~\ref{thm:schurB}, by noticing that $\SYT(\lambda)$ can be identified with $\BSYT(\emptyset,\lambda)$, and then with $\SDT(\mu)$, where  $(\emptyset,\lambda)=\psi(\mu)$ and $\mu$ is an empty 2-core, see Section~\ref{sec:Littlewood}.

%Hence, for every $w\in S_n$,
%the Chow's quasi-symmetric function $F^B_{\Des_B(w)}$ is equal to
%the Gessel's quasi-symmetric function $F_{\Des(w)}$.
\end{proof}

\begin{proof}[Proof of Proposition~\ref{cor:Chow_FC}.]
First notice that for every $w\in S_n$, $\bl(w)=\bl(w^{-1})$,
while for $w\in B_n\setminus S_n$ this is not necessarily the case. 
Combining this with   
%by Theorem~\ref{thm:ABR} 
%\[
%\sum\limits_{w \in \FC(S_n)} q^{\bl(\pi)} F_{\Des(\pi)}=\sum\limits_{k=0}^{\lfloor n/2\rfloor}
%\left(\sum\limits_{j=0}^n  a_{n,k,j} \ q^j\right) s_{(n-k,k)},
%\]
Theorem~\ref{thm:main}, 
Theorem~\ref{thm:ABR} and Lemma~\ref{lem_GC}, we obtain
\begin{eqnarray*}
\sum\limits_{w \in \FC(B_n)%\setminus \FC(S_n)
} q^{\bl(w^{-1})} F^B_{\Des_B(w)}
&=&\sum\limits_{w\in \FC(B_n)\setminus \FC(S_n)} q^{\bl(w^{-1})} F^B_{\Des_B(w)}+\sum\limits_{w\in \FC(S_n)} q^{\bl(w^{-1})} F^B_{\Des_B(w)}\\
&=&
\sum\limits_{k=1}^{n} \left(\sum\limits_{j=0}^n b_{n,k,j} q^j \right)
s_{(k)}(x_1,x_2,\ldots)\ s_{(n-k)}(x_0,x_1,\ldots)\\
&&\hskip1cm+\sum\limits_{k=0}^{\lfloor n/2\rfloor}
\left(\sum\limits_{j=0}^n  a_{n,k,j} \ q^j\right) s_{(n-k,k)}(x_0,x_1,\dots),
\end{eqnarray*}
with  non-negative integer coefficients $b_{n,k}$ and $a_{n,k}$. Equivalently,
\begin{multline*}
\Q^C(\{w\in\FC(B_n)|\ \bl(w^{-1})=j\})=
\sum\limits_{k=1}^{n}  b_{n,k,j}\ 
s_{(k)}(x_1,x_2,\ldots)\ s_{(n-k)}(x_0,x_1,\ldots)\\
+\sum\limits_{k=0}^{\lfloor n/2\rfloor}  a_{n,k,j}\   s_{(n-k,k)}(x_0,x_1,\dots).
\end{multline*}
By Propositions~\ref{prop:domino-chow} and~\ref{thm:schurB}, 
for every $\lambda\in P^{0}(n)$, ${\mathcal G}_\lambda=s_{\lambda^{-}}(x_1,x_2,\ldots)\ s_{\lambda^{+}}(x_0,x_1,\ldots)$, thus the right hand side is a non-negative
sum of domino functions, completing the proof.
\end{proof}
%\bj{\todo{Check if the partition above is $(n-k,k)$ or $(k,n-k)$.}}

\bigskip

Another definition of type $B$ Schur-positivity was suggested in~\cite{AAER}, using Poirier's type $B$ quasi-symmetric functions,
which were introduced in~\cite{Po}. The following definition reformulates~\cite{Po, AAER}.
Let ${\X}:=(x_1,x_2,\ldots)$ and ${\Y}:=(y_1,y_2,\ldots)$ be two infinite sets of formal variables.

\begin{definition} We define the following. 
\begin{itemize}
    \item[1.] 
Let $<_r$ be the order on $[\pm n]$ 
\[
-1 <_r -2<_r\cdots<_r-n<_r1<_r2<_r\cdots<_rn.
\]
The {\em r-descent set} of $w\in B_n$ is 
%defined as a pair $(S(w),\varepsilon(w))$, where
\[
\wDes(w):=\{1\le i< n|\ w_i>_r w_{i+1}\}.
\]
%and, for all $1\le i\le n$, $\varepilon(w)_i=-$ if $w_i<0$ and  $+$ otherwise.
%\end{definition}

%Notice that the signed descent set determines $\Des_B$.

%\begin{definition}
\item[2.] The {\em Poirier type $B$ quasi-symmetric function}, associated with
$w\in B_n$ is
\[
    F^P_w (\X, \Y) \ :=
    \sum_{\substack{1\le i_1 \le i_2 \le \ldots \le i_n \\ j \in \wDes(\sigma) \,\Rightarrow\, i_j < i_{j+1}}}
    z_{i_1} z_{i_2} \cdots z_{i_n}
\]
where $z_{i_j} = x_{i_j}$ if $j\not\in\Neg(w)$, and $z_{i_j} = y_{i_j}$ 
if $j\in \Neg(w)$. 
\smallskip 

\item[3.] For a subset $A\subseteq B_n$ let 
\[
\Q^P(A):=\sum\limits_{w\in A} F^P_w(\X,\Y).
\]
A subset $A\subseteq B_n$ is {\em Poirier type $B$ Schur-positive}
if $\Q^P(A)$ is symmetric in $\X,\Y$ and can be expanded as a non-negative sum in the basis  $s_\lambda(\X)  s_\mu(\Y)$. 
\end{itemize}
\end{definition}
\begin{example}
Let $w=[-3,-1,2]\in B_3$. Then $\wDes(w)=\{1\}$ and thus $F_w^P(X,Y)=\sum\limits_{i_1 < i_2 \leq i_3}y_{i_1}y_{i_2}x_{i_3}$. 

\end{example}
\begin{remark}\label{rem:P-equid}
%\begin{itemize}
 %   \item[1.] 
 The {\em signed descent set} of a signed permutation $w\in B_n$
is the pair $(\wDes(w),\Neg(w))$.
%is a pair $(\wDes(w), \varepsilon(w))$,
%where  $\varepilon(w)=(\varepsilon_1,\dots,\varepsilon_n)$ 
%with $\varepsilon_i:=+$ if $w_i>0$ and $-$ otherwise.
The {\em signed descent set} of a bi-tableau $T=(T^-,T^+)$ of bi-shape $(\lambda^-, \lambda^+)$ 
is the pair $(\Des(T), \Neg(T))$
%is a pair $(\Des(\sT), \varepsilon(\sT))$,
where  $\Des(T)$ is the descent set of $T$ defined in~\eqref{eq:bi_st}, and $\Neg(T)$ is the set of entries in $T^-$.
By~\cite[Cor. 3.7]{AAER}, 
a subset $A\subseteq B_n$ is 
Poirier type $B$ Schur-positive if and only if the distribution of the signed descent set
%equidistribution of the signed descent set 
over $A$ is equal to its distribution over bi-tableaux of some multiset of bi-shapes. 
Furthermore, in this case, 
%By the proof of~\cite[Cor. 3.7]{AAER},
\[
\Q^P(A)=\sum\limits_{\lambda\in P^0(n)}%c_{\lambda^-,\lambda^+} 
c_\lambda\  s_{\lambda^-}(\X) s_{\lambda^+}(\Y)
\]
if and only if
\[
\sum\limits_{w\in A} {\bf x}^{\wDes(w)}{\bf y}^{\Neg(w)}=
\sum\limits_{\lambda\in P^0(n)} c_\lambda \sum\limits_{T\in \BSYT(\lambda^-,\lambda^+)}
{\bf x}^{\Des(T)}{\bf y}^{\Neg(T)}.
\]
Here we use the notation from Section~\ref{sec:2.3}, $P^0(n)$ 
for the set of partitions of $2n$ with empty 2-core, and
$(\lambda^-,\lambda^+)$  
for the $2$-quotient of a partition $\lambda\in P^0(n)$.
\end{remark}

%\bj{\todo{
%It seems that $\lambda^+$ and $\lambda^-$ are the partitions defined in the previous sections (with Littlewood's map) but they aren't, right ??? 
%If this is the case, it is better to change notation.}}

%\todo{YR: see above clarification. }

%\bj{\todo{
%Add that $c_\lambda \geq 0$.}}

%\todo{YR:
%It is written "in this case", namely, when  $A\subseteq B_n$ is 
%Poirier Schur-positive. In fact, this assumption is not needed, but I have %no explicit reference for that,
%and we do not need the more general statement.}

%\bj{\todo{
%It is not clear how to pass form the variables $\X$ to ${\bf x}$.}}
%\todo{YR: This is explained in the proof of~\cite[Cor. 3.7]{AAER}. I do not think that I should reprove this theorem.}

\begin{remark}\label{rem:P-examples}
Examples of Poirier type $B$ Schur-positive sets include
conjugacy classes~\cite[Theorem 16]{Po} and inverse signed descent classes $\{w\in B_n|\ \Des(w^{-1})=I, \Neg(w^{-1})=J\}$ ~\cite[Proposition 5.5.1]{AAER}.
For more examples see~\cite{AAER}.
\end{remark}

\begin{proposition}\label{prop:Porier_FC}
For every $n>2$, $\FC(B_n)$ is not Poirier type $B$ Schur-positive.
\end{proposition}

\begin{proof}
%Poirier Schur-positivity of $A\subseteq B_n$ is equivalent to an equidistribution of the signed descent set over $A$ with standard bi-tableaux of a multiset of bi-shapes~\cite[Cor. 3.7]{AAER}. 
Observe that for any $\lambda\vdash n-1$ the number of standard bi-tableaux of bi-shape $((1),\lambda)$ with $\Neg(T)=\{i\}$
is independent of $i$.
Combining this with Remark~\ref{rem:P-equid}, we deduce that for every Poirier type $B$ Schur-positive set $A\subseteq B_n$, 
the cardinality of the set
$\{w\in A \mid \Neg(w)=\{i\}\}$ is independent of $i$. For $n\ge 3$, the set $\FC(B_n)$ violates this condition as follows.
By Proposition~\ref{prop:fc1}, $\#\{w\in \FC(B_n) \mid \Neg(w)=\{n\}\}=n$, since $w=[w_1,\dots,w_n]\in \FC(B_n)$ avoids a decreasing subsequence of order 3,
thus for every $j$, the only signed permutation in  $\FC(B_n)$ with $w_n=-j$ is $[1,2,\dots,j-1,j+1,\dots,n,-j]$. On the other hand,
$\#\{w\in \FC(B_n)\mid \Neg(w)=\{1\} \}\ge 2n-2$, since for every $1\le j< n-1$ 
there are at least two signed permutations in $\FC(B_n)$ with $w_1=-j$,  $[-j,1,2,\dots,j-1,j+1,\dots,n-2,n,n-1]$ and $[-j,1,2,\dots,j-1,j+1,\dots,n]$,
%for every $1\le j\le n$, and, which 
the latter is in $\FC(B_n)$ for $j=n-1,n$ as well. 
\end{proof}

%\todo{\bj{Do we want to transform the next Proposition into a Theorem and add at least part 1 in the section Main Results of the introduction ?}}

\begin{theorem}\label{lem:PC} We have the following.
%\todo{YR: Poirier's implies Chow.}
\begin{itemize}
\item[1.] 
 Every Poirier type $B$ Schur-positive set $A$ is a Chow type $B$ Schur-positive set.
 
 \item[2.] In this case,
% Let $\Q^P(A)$ and $\Q^C(A)$  be the Poirier's and Chow's quasisymmetric functions of a Poirier Schur-positive subset $A\subseteq B_n$. 
if
\[
\Q^P(A)=\sum\limits_{\lambda\in P^0(n)}c_{\lambda}\  s_{\lambda^-}(\X) s_{\lambda^+}(\Y)
\]
then
\[
\Q^C(A)=\sum\limits_{\lambda\in P^0(n)}c_{\lambda}\ s_{(\lambda^-)'}(x_1,x_2,\dots) s_{\lambda^+}(x_0,x_1,\dots),
\]
%\bj{\todo{$x_0$ is on $s_\lambda^{+}$. YR: corrected, thanks!}}
%there exist expansions of $Q_P(A)$ and $Q_C(A)$ such that the coefficient of  $s_\mu(x_0,x_1,\dots) s_\lambda(x_1,x_2,\dots)$
%in the expansion of $Q_P(A)$  in Poirier's fundamental quasi-symmetric functions is equal to
%to the coefficient of $s_{\mu'}(x_0,x_1,\dots) s_\lambda(x_1,x_2,\dots)$
%in the expansion of $Q_C(A)$ in Chow's fundamental quasi-symmetric functions,
%(same coefficients), 
where $(\lambda^-,\lambda^+)$ is the 2-quotient of $\lambda$, and $(\lambda^-)'$ is the conjugate partition of $\lambda^-$.
 %   \item[2.]  In particular, every Poirier's type $B$ Schur-positive set is Chow's type $B$ Schur-positive set.
\end{itemize}
\end{theorem}

\begin{proof}
Let $A\subseteq B_n$ be a type $B$ Poirier Schur-positive set. 
%The {\em signed descent set} of a signed permutation $w\in B_n$
%is the pair $(\wDes(w),\Neg(w))$.
%%is a pair $(\wDes(w), \varepsilon(w))$,
%%where  $\varepilon(w)=(\varepsilon_1,\dots,\varepsilon_n)$ 
%%with $\varepsilon_i:=+$ if $w_i>0$ and $-$ otherwise.
%The {\em signed descent set} of a bi-tableau $\sT=(\sT^-,\sT^+)$ of bi-shape $(\lambda^_, \lambda^+)$ 
%is the  $(\Des(\sT), \Neg(\sT))$
%is a pair $(\Des(\sT), \varepsilon(\sT))$,
%where  $\Des(\sT)$ is the standard descent set of $\sT$ (see Eq.~\eqref{eq:bi_st}) and $\Neg(\sT)$ is the set of entries in $\sT^-$.
%and $\varepilon(w)=(\varepsilon_1,\dots,\varepsilon_n)$ 
%with $\varepsilon_i:=+$ if $i\in T^+$  and $-$ if $i\in T^-$.
 By definition,
\[
\Q^P(A)=\sum\limits_{\lambda\in P^0(n)}c_{\lambda}\  s_{\lambda^-}(\X) s_{\lambda^+}(\Y)
\]
with non-negative integer coefficients $c_\lambda\ge 0$. 
%then
%$A\subseteq B_n$ is Poirier Schur-positive then 
By Remark~\ref{rem:P-equid}, 
\begin{equation}\label{eq:111}
\sum\limits_{w\in A} {\bf x}^{\wDes(w)}{\bf y}^{\Neg(w)}=
\sum\limits_{\lambda\in P^0(n)} c_{\lambda} \sum\limits_{T\in \BSYT(\lambda^-,\lambda^+)}
{\bf x}^{\Des(T)}{\bf y}^{\Neg(T)}.
\end{equation}
Let  $T'=((T^-)',T^+)\in \BSYT((\lambda^-)',\lambda^+)$ be the standard Young bi-tableau obtained by transposing $T^-$. 
Note that %$\Neg(T')=\Neg(T)$, in particular,
%\[
%0\in \Des_B(T)\Longleftrightarrow 
%1\in T^- \Longleftrightarrow 1\in (T^-)'\Longleftrightarrow 
%0\in \Des_B(T'),
%\]
%Also, 
for every $w\in B_n$ %and a bi-tableau $T$, 
\[
\Des(T)=\wDes(w)\ \mbox{and} \  \Neg(T)=\Neg(w) \Longleftrightarrow \Des(T')=\Des(w) \ \mbox{and} \ \Neg(T')=\Neg(w).
\]
We conclude that Equation~\eqref{eq:111} is equivalent to the following.
%\begin{equation}\label{eq:112}
%\sum\limits_{w\in A} {\bf x}^{\Des(w)}{\bf y}^{\Neg(w)}=
\begin{eqnarray*}
\sum\limits_{w\in A} {\bf x}^{\Des(w)}{\bf y}^{\Neg(w)}&=&\sum\limits_{\lambda\in P^0(n)} c_{\lambda} \sum\limits_{T\in \BSYT(\lambda^-,\lambda^+)}
{\bf x}^{\Des(T')}{\bf y}^{\Neg(T')}\\
&=&\sum\limits_{\lambda\in P^0(n)} c_{\lambda} \sum\limits_{T\in \BSYT((\lambda^-)',\lambda^+)}
{\bf x}^{\Des(T)}{\bf y}^{\Neg(T)}.
\end{eqnarray*}
%\end{equation}
%is equivalent to an equidistribution of the signed descent set over $A$ with bi-tableaux of a multiset of bi-shapes. 
%For every $T=(T^-,T^+)\in \BSYT(\lambda^-,\lambda^+)$ let 
%where $T'=((T^-)',T^+)\in \BSYT((\lambda^-)',\lambda^+)$ is the standard Young bi-tableau obtained by transposing $T^-$. 
%This map determines a bijection from $\BSYT(\lambda^-,\lambda^+)$ to  %$\BSYT((\lambda^-)',\lambda^+)$, which satisfies
%\[
%\wDes(T)=\Des(T')
%\]
%since Finally, notice that the signed descent set is determined with respect to the order
%$-1<-2<\cdots<-n<-1<-2<\cdots<n$, while the Coxeter descent set is determined w.r.t. to the order
%$-n<\cdots<-1<-1<-2<\cdots<n$ (see Definition 5). Transposing the negative tableaux reverse the order.
%and 
%Applying this bijection, formula
Setting $y_1=x_0$ and $y_2=\cdots=y_n=1$ we obtain 
%, Equation~\eqref{eq:111} implies
\[
\sum\limits_{w\in A} {\bf x}^{\Des_B(w)}=
\sum\limits_{\lambda\in P^0(n)} c_{\lambda} \sum\limits_{T\in \BSYT((\lambda^-)',\lambda^+)}
{\bf x}^{\Des_B(T)},
\]
%Denote by $\psi$ the $2$-quotient bijection from $P^0(n)$ to pair of partitions of size $n$ (in other words $\psi$ is the Littlewood correspondence with $r=2$).
By Lemma~\ref{lemma:MV}, for every $\lambda\in P^0(n)$ the distribution of $\Des_B$ over $\SDT(\lambda)$ is equal to its distribution
over $\BSYT(\psi(\lambda))$, where $\psi$ is the Littlewood decomposition defined in Section~\ref{sec:Littlewood}. 
Combining this with Propositions~\ref{prop:domino-chow} and~\ref{thm:schurB}, the last equation then implies 
\begin{eqnarray*}
\Q^C(A)&=&\sum\limits_{w\in A}F^B_{\Des_B(w)}=
\sum\limits_{\lambda\in P^0(n)} c_{\lambda} \sum\limits_{T\in \BSYT((\lambda^-)',\lambda^+)}
F^B_{\Des_B(T)}=\sum\limits_{\lambda\in P^0(n)} c_{\lambda} \sum\limits_{\sT\in \SDT(\psi^{-1}((\lambda^-)',\lambda^+)}F^B_{\Des_B(\sT)}\\
&=&\sum\limits_{\lambda\in P^0(n)} c_\lambda 
{\mathcal G}_{\psi^{-1}((\lambda^-)',\lambda^+)}=
\sum\limits_{\lambda\in P^0(n)}c_{\lambda}\ s_{(\lambda^-)'}(x_1,x_2,\dots) s_{\lambda^+}(x_0,x_1,\dots),
\end{eqnarray*}
as desired.
%On the other hand, 
%Chow Schur-positivity is equivalent to an equidistribution of the Coxeter descent set $\Des_B$ with domino tableaux of a given multiset of shapes
%[reference ???].
%namely, $A$ is type $B$ Chow Schur-positive.
\end{proof}

%\begin{corollary}\label{cor:PC}
%Every Poirier's type $B$ Schur-positive set is Chow's type $B$ Schur-positive set.
%\end{corollary}

%\begin{proof}
%Combine ?? with ??.
%\end{proof}

\begin{remark}
Combining Remark~\ref{rem:P-examples} with Theorem~\ref{lem:PC}.1, 
%Corollary~\ref{cor:PC},
one deduces that conjugacy classes and inverse signed (or unsigned) descent classes 
in $B_n$ are Chow type $B$ Schur-positive.
\end{remark}

\begin{remark}\label{rem:converse}
The converse of %Corollary~\ref{cor:PC} 
 Theorem~\ref{lem:PC}.1 
does not hold. Indeed, 
by Proposition~\ref{cor:Chow_FC}, $\FC(B_n)$ is Chow type $B$ Schur-positive,
while by Proposition~\ref{prop:Porier_FC},
 %$\FC(B_n)$ 
 it is not Poirier type $B$ Schur-positive. 
%Note that examples for Chow's type $B$ Schur-positivity, given in~\cite{MV4}, are all Poirier Schur-positive. 
\end{remark}

%\bj{\todo{Can we replace easily the block number with the length function ? }}

%\todo{YR: Indeed, it is an interesting team to find other Schur-positive statistics on $\FC(W)$ for $W=S_n,B_n$. Clearly $\Des(\pi^{-1})$ works, but length does not work, as it. 
%For example, the set of elements in $FC(W)$ with prescribed length is not Schur positive.  
%For example, $\{\pi\in \FC(S_4)|\ \ell(\pi)=4 \}=\{[3,4,1,2]\}$, which is not. Maybe some modifications work, but not obvious ones. It could be a good idea to add some open problems; if I am not wrong, this one essentially appears in last section of~\cite{ABR}.}

%%%%%%%%%%%%%%%%%%%%%%%%%%%%%%%%%%%%%%%%%%%%%%%%%

\end{document}